\documentclass[11pt]{amsart}
\pdfoutput=1

\pdfpageattr{/Group <</S /Transparency /I true /CS /DeviceRGB>>}

\usepackage[T1]{fontenc}
\usepackage[english]{babel}
\usepackage{geometry}
\usepackage{amsfonts,amsmath,amssymb,amsthm,mathrsfs,mathtools}
\usepackage{verbatim}
\usepackage[vlined,oldcommands,linesnumbered]{algorithm2e}
\SetKwInOut{Input}{input}
\SetKwInOut{Output}{output}
\SetKw{Or}{or}
\SetArgSty{rm}
\SetFuncSty{tt}
\setalcapskip{1ex}

\usepackage{tikz}
\usetikzlibrary{calc}
\usepackage{caption}
\usepackage[labelformat=parens]{subcaption}
\usepackage{xspace}
\usepackage{booktabs,tabularx}
\usepackage{dcolumn}
\usepackage{marvosym}
\usepackage{microtype}
\usepackage{enumerate}
\usepackage{mathtools}
\usepackage{hyperref}
\usepackage[capitalise]{cleveref}
\usepackage[autostyle=true]{csquotes}

\microtypesetup{tracking,kerning,spacing}
\microtypecontext{spacing=nonfrench}

\hypersetup{colorlinks,linkcolor=blue,citecolor=blue,urlcolor=blue}

\newcommand\RR{{\mathbb R}}

\newcommand{\cB}{\mathcal B}

\newcommand\SetOf[2]{\left\{\left.#1\vphantom{#2}\ \right|\ #2\vphantom{#1}\right\}}
\newcommand\smallSetOf[2]{\{{#1}\,|\,{#2}\}}
\DeclareMathOperator{\conv}{conv}

\DeclareMathOperator{\rank}{rk}
\DeclareMathOperator{\size}{\#}

\DeclareMathOperator{\dist}{dist}

\theoremstyle{plain}
    \newtheorem{theorem}{Theorem}
    \newtheorem{corollary}[theorem]{Corollary}
    \newtheorem{lemma}[theorem]{Lemma}
    \newtheorem{proposition}[theorem]{Proposition}
\theoremstyle{definition}
    \newtheorem{remark}[theorem]{Remark}
    \newtheorem{example}[theorem]{Example}

    \newtheorem{problem}{Problem}
    
    \newtheorem{question}{Question}

\definecolor{myred}{rgb}{0.7,0.0,0.2}
\definecolor{myblack}{rgb}{0,0,0}
\definecolor{myred2}{rgb}{0.7,0.0,0.2}
\definecolor{myblue}{rgb}{0,0.2,0.7}
\definecolor{myred3}{rgb}{0.7,0,0.2}
\definecolor{myred4}{rgb}{0,0,0}

\title{Reconstructibility of matroid polytopes}

\author{Guillermo Pineda-Villavicencio}
\author{Benjamin Schr\"oter}

\address{
   School of Information Technology, Deakin University, Geelong, Australia
   and
   Centre for Informatics and Applied Optimisation, Federation University, Australia
}
\email{work@guillermo.com.au}

\address{
  Department of Mathematics, KTH Royal Institute of Technology, Stockholm, Sweden
}
\email{schrot@kth.se}


\begin{document}

\begin{abstract}
We specify what is meant for a polytope to be reconstructible from its graph or dual graph. And we introduce the problem of class reconstructibility, i.e., the face lattice of the polytope can be determined from the (dual) graph within a given class. 
We provide examples of cubical polytopes that are not reconstructible from their dual graphs.
Furthermore, we show that matroid (base) polytopes are not reconstructible from their graphs and not class reconstructible from their dual graphs;  our counterexamples include hypersimplices.
Additionally, we prove that matroid polytopes are class reconstructible from their graphs, and we present a $O(n^3)$ algorithm that computes the vertices of a matroid polytope from its $n$-vertex graph. Moreover, our proof includes a characterisation of all matroids with isomorphic basis exchange graphs.  
\end{abstract}
\maketitle

\section{Overview}
\noindent This manuscript deals with several fundamental structures of discrete geometry, namely polytopes, their face structure, and in particular their graphs, as well as matroids. We mainly discuss whether the graph of a matroid polytope determines its matroid and the entire polytope.

Polytopes and polyhedral structures have found applications in a growing number of areas, ranging from operations research, machine learning and statistics to topology and algebraic geometry.  We refer to \cite{GoodmanORourkeToth:2018} for a significant number of examples.  
Most prominently, polytopes arise as sets of feasible solutions of linear programs \cite[Ch.\,49]{GoodmanORourkeToth:2018}, and more generally, of convex programs. 

In many of the above applications, data or solution spaces are high-dimensional convex polytopes or can be approximated by these. Given that high-dimensional data is often hard to analyse or visualise, most methodologies include dimension reduction techniques (see \cite{burges2010dimension}), which try to embed the data in a low dimensional space. Unfortunately, dimension reductions in many cases carry a subsequent loss of information. This paper partly addresses this phenomenon. We investigate polytope classes whose combinatorial structure can be entirely reconstructed from $1$-dimensional objects, their graphs, which are well studied and often easier to understand. With this knowledge and relevant reconstruction algorithms, combinatorial questions on high-dimensional polytopes can be answered by examining their graphs. Graph reconstruction means algorithmically  that, given the dimension of a polytope, and either the graph or dual graph of the polytope, it can be decided whether a given induced subgraph is the graph of a facet.
Furthermore, the relation between polytopes and their graphs is central for the simplex method, a standard algorithm in linear programming. It is an open problem if there exist a pivot rule such that the method becomes strongly polynomial and thus solves Smale's 9th problem.

We focus our attention on matroid (base) polytopes and their graphs. Matroids themselves are a (combinatorial) generalisation of graphs and (linear) independence in vector spaces. They have many applications in discrete and algorithmic geometry; see \cite{Oxley:2011} and \cite{White:1986} for further details.
Matroid polytopes are one of many cryptomorphic ways to encode a matroid. Matroid polytopes, their regular subdivisions, and their normal fans play a fundamental role in tropical geometry; see \cite{MaclaganSturmfels:2015} and \cite{Joswig:2020} for further details.
It is remarkable that the three-dimensional faces of matroid subdivisions determine tropical linear spaces entirely \cite[Theorem 14]{OlartePanizzutSchroeter}. Here we discuss the related case of matroid polytopes and their one-dimensional faces, i.e., their vertex-edge graphs.
These graphs are meaningful as their edges correspond to element exchanges between bases, i.e., maximal independent sets of a matroid.
They are studied actively in connection with White's 1980 conjecture about symmetric basis exchanges \cite{White:1980}. Besides,  
 graphs of matroid polytopes have been characterised by Maurer \cite{Maurer:1973}. This characterisation is crucial to our work.

An important subclass of matroid polytopes, which we investigate further, are hypersimplices; these are the (base) polytopes of uniform matroids. 
Their graphs are known as Johnson graphs and play an important role in the theory of error correcting codes; see for example \cite{MacWilliamsSloane}.
The term ``hypersimplex'' was introduced in \cite{GabGelLos75}, but these polytopes appeared much earlier, for example in the work of Coxeter \cite[p.\,163]{Cox73} where he uses a construction of Stott \cite[p.\,15]{Stott:1910}. Since then, hypersimplices have featured in a number of applications, including combinatorial optimisation. 
Matroid polytopes, and in particular hypersimplices, naturally appear as moment polytopes of Grassmannians under torus actions \cite{GelfandEtAl:1987}. This fact connects tropical algebraic geometry to high energy physics and scattering amplitudes. Here positive geometries,  hypersimplices, and other positroids are central; see \cite{Postnikov:2018} and \cite{ABCGPT:2016}. Hypersimplices are also weight polytopes of the fundamental representations of the general linear group.

This manuscript is organised as follows. Beside this general introduction, it consists of three further sections. In Section~\ref{sec:intro-rec}, we introduce the classical problem of reconstructing a polytope from its (dual) graph (Problem~\ref{prob:rec}), or more generally, from its $k$-skeleton. We also introduce a variation of the classical reconstruction in Problem~\ref{prob:class-rec}, which we call \emph{class reconstruction problem}. This new problem is inspired by a conjecture of Joswig,   \cite[Conjecture 3.1]{Joswig:2000}, stating that a cubical polytope is reconstructible from its dual graph. We examine this conjecture in this section and show that it is false in the sense of \cref{prob:rec} by providing a counterexample in Example~\ref{ex:dualcubical}. However, it remains open for the problem of class reconstructibility of cubical polytopes; see \cref{prob:class-rec}. In the remaining part of the paper, we  focus on the reconstructibility of matroid (base) polytopes from their graphs.

In Section~\ref{sec:BasisExchangeGraphs}, we exploit Maurer's characterisation \cite{Maurer:1973} of basis exchange graphs, i.e., the vertex-edge graphs of matroid polytopes, to obtain Theorem~\ref{thm:recmatroids}, which characterises matroids with isomorphic basis exchange graphs. This part does not use any polytope theory. In Section~\ref{sec:RecMatroidPolytopes}, we present Algorithm~\ref{algo:label}, which, given a graph on $n$ nodes,  computes the bases of a matroid, or equivalently the vertices of its matroid polytope,  whenever the graph is the basis exchange graph of a matroid.  Furthermore, Algorithm~\ref{algo:label} returns \emph{false} if the graph is not that of a matroid polytope. This algorithm terminates after at most $O(n^3)$ steps.
Combining \cref{algo:label} and Theorem~\ref{thm:recmatroids}, we derive that matroid polytopes are class reconstructible from their graphs (Theorem~\ref{thm:rec-matroid-polys}) and that they are neither class reconstructible from their dual graphs (Corollary~\ref{cor:rec-dual-matroid-polys}) nor reconstructible from their graphs  (Proposition~\ref{prop:counterexample}).

\section{The problem of polytope reconstruction}\label{sec:intro-rec}
\noindent 
A (convex) {\it $d$-polytope} is the convex hull of finitely many points whose affine hull is $d$-dimensional. The $k$-dimensional {\it skeleton}, or $k$-skeleton,  of a polytope~$P$ is the set of all its faces of dimension at most~$k$. Its $1$-skeleton is its {\it (vertex-edge) graph} $G(P)$. The {\it combinatorial structure} of a polytope is given by its {\it face lattice}, a partial ordering of the faces by inclusion. Two polytopes are \emph{combinatorially equivalent} if there is an inclusion-preserving bijection between their face lattices.
In this article,  we do not distinguish between combinatorially equivalent polytopes.

For a polytope $P$ that contains the origin in its relative interior, the polar polyhedron of $P$ is  defined as 
\[
\SetOf{y\in \RR^d}{x\cdot y\le 1 \text{ for all $x$ in $P$}}.
\]
This polyhedron might be unbounded, which  is the case whenever $P$ is not full dimensional. We intersect the polar polyhedron with the affine hull of $P$. This intersection is a polytope~$P^*$ called the {\it polar polytope} of $P$. If $P$ does not contain the origin, we translate the polytope so that it does to obtain the polar polytope. Note that, as long as the origin is in the relative interior of $P$, translating the polytope $P$ changes the geometry of $P^*$, but not the face structure. It is also the case that the face lattice of $P^*$ is the inclusion-reversed face lattice of $P$. In particular, the vertices of $P^*$ correspond to the \emph{facets} of $P$, i.e., the codimension-one faces of $P$.
The \emph{dual graph} of a polytope $P$ is the graph of the polar polytope, or equivalently, the graph  on the set of facets of $P$ where two facets are adjacent in the dual graph if they share a \emph{ridge}, a codimension-two face. 

The classical graph reconstruction problem can be stated as follows.
\begin{problem}[Reconstruction problem]\label{prob:rec} Given a parameter $d$, and the (dual) graph of a $d$-polytope, determine the face lattice of the polytope.  
\end{problem}

Problem~\ref{prob:rec} requires that there is exactly one $d$-polytope whose (dual) graph is the given graph. In practice, it is unclear how to obtain the face lattice from a graph, even if there is only one $d$-polytope with this graph.
Theoretically, one can enumerate the face lattices of all $d$-polytopes with a given number of vertices \cite[p.\,91]{Gruenbaum:2003}. Hence, theoretically, we are able to decide whether a polytope is reconstructible from its graph or not, and if it is, we can reconstruct the face lattice. However, Gr\"unbaum's approach is impractical \cite[p.\,96b]{Gruenbaum:2003} as it relies on Tarski's decidability of expressions in first-order logic over the reals. 

The reconstruction problem can plainly be  extended to skeleta of polytopes; see Gr\"unbaum's book on convex polytopes \cite[Chaper 12]{Gruenbaum:2003} or Bayer's survey on the topic \cite{Bayer:2018}. For example every $d$-polytope is reconstructible from its $(d-2)$-skeleton \cite[Theorem~12.3.1]{Gruenbaum:2003}, which means that you can construct the full face lattice from the $(d-2)$-skeleton.  But there are $d$-polytopes with isomorphic $(d-3)$-skeleton, e.g., a bipyramid over a $(d-1)$-simplex and the pyramid over a bipyramid over a $(d-2)$-simplex have equivalent $(d-3)$-skeleta.

It is well known that the graph of a polytope does not in general determine its dimension.
For example, the $k$-skeleton of the cyclic $d$-polytope on $n$ vertices \cite[page~11]{Ziegler:1995} agrees with the $k$-skeleton of the $(n-1)$-simplex, for every natural numbers $n$, $d$ and $k$ satisfying $n>d\geq 2k+2 \geq 4$; see \cite[Chapter 12]{Gruenbaum:2003}; in particular, the graph of the cyclic $4$-polytope on $n>4$ vertices is the complete graph on $n$ vertices.

It is a famous result of Blind and Mani \cite{BlindMani:1987}, and later Kalai \cite{Kalai:1988}, that every {\it simple $d$-polytope}, a polytope in which every vertex is incident to precisely $d$ edges, can be reconstructed from its graph. That is, the graph of a simple $d$-polytope $P$ determines the entire face structure of $P$ when the dimension of $P$ is given. Other recent reconstruction developments can be found in \cite{DoolittleEtAl:2017}, \cite{PinedaEtAl:2017}, and \cite[Section~19.5]{GoodmanORourkeToth:2018}.

\begin{figure}[t!]
	\caption{Schlegel projections of the $4$-polytopes $Q$ and its polar $Q^*$ from Example~\ref{ex:dualcubical}. The polar polytope $Q^*$ has the same dual graph as the $4$-cube. } 
\label{fig:p-cube-dual-graph}
\begin{subfigure}[t]{0.45\textwidth}
    \centering
    \subcaption*{A Schlegel projection of $Q$.}

\begin{tikzpicture}[x  = {(-0.3cm,-0.7cm)},
                    y  = {(-0.42cm,-0.4cm)},
                    z  = {(0.85cm,-0.4cm)},
                    scale = .77,
                    color = {lightgray}]

  \tikzstyle{pointcolor_x} = [fill=myred]
  \tikzstyle{pointcolor_p} = [fill=black]

  \coordinate (v0_p) at (-2, -2, -2);
  \coordinate (v1_p) at (0.5, -0.8, 0.2);
  \coordinate (v2_p) at (-2, -2, 4);
  \coordinate (v3_p) at (4, -2, -2);
  \coordinate (v4_p) at (0.3, -1.3, 0.2);
  \coordinate (v5_p) at (3.5, -5, 4);
  \coordinate (v6_p) at (-2, 4, -2);
  \coordinate (v7_p) at (0.5, -3.2, 2.5);

  \tikzstyle{linestyle_x} = [preaction={draw=white, line cap=round, line width=2.5 pt}, line cap=round, line join=round, color=myred, line width=1 pt]
  \tikzstyle{linestyle_p} = [preaction={draw=white, line cap=round, line width=2.5 pt}, line cap=round, line join=round, color=black, ultra thick]

  \draw[color=white] (-6.6,0,0) -- (3.5,0,0);
  \draw[linestyle_p] (v0_p) -- (v2_p) -- (v6_p) -- cycle;
  \draw[linestyle_p] (v5_p) -- (v6_p);
  \draw[linestyle_p] (v2_p) -- (v5_p) -- (v3_p) -- (v6_p);
  \draw[linestyle_x] (v2_p) -- (v1_p) -- (v6_p);
  \draw[linestyle_x] (v2_p) -- (v4_p) -- (v6_p) (v0_p) -- (v4_p) -- (v1_p);
  \draw[linestyle_x] (v0_p) -- (v7_p) -- (v2_p) (v1_p) -- (v7_p) -- (v4_p);
  \draw[linestyle_x] (v3_p) -- (v1_p) -- (v5_p);
  \draw[linestyle_x] (v3_p) -- (v4_p);
  \draw[linestyle_x] (v3_p) -- (v7_p) -- (v5_p);
  \draw[linestyle_p] (v0_p) -- (v3_p) -- (v5_p) -- cycle;

  \fill[pointcolor_p] (v0_p) circle (4 pt);
  \fill[pointcolor_x] (v1_p) circle (3 pt);
  \fill[pointcolor_p] (v2_p) circle (4 pt);  
  \fill[pointcolor_p] (v3_p) circle (4 pt);
  \fill[pointcolor_x] (v4_p) circle (3 pt);
  \fill[pointcolor_p] (v5_p) circle (4 pt);
  \fill[pointcolor_p] (v6_p) circle (4 pt);
  \fill[pointcolor_x] (v7_p) circle (3 pt);
    
\end{tikzpicture}
\end{subfigure}
\begin{subfigure}[t]{0.45\textwidth}
    \centering
    \subcaption*{A Schlegel projection of $Q^*$.}

\begin{tikzpicture}[x  = {(-0.3cm,-1cm)},
                    y  = {(0.8cm,-0.40cm)},
                    z  = {(0.6cm,0.1cm)},
                    scale = 8,
                    color = {lightgray}]

  \tikzstyle{pointcolor_x} = [fill=myred]
  \tikzstyle{pointcolor_p} = [fill=black]
  
  \coordinate (v0_q) at (0, 0.408248, 0.144338);
  \coordinate (v1_q) at (0.101015, 0.174964, 0.061859);
  \coordinate (v2_q) at (0.215206, -0.124249, -0.0815821);
  \coordinate (v3_q) at (-0.176777, -0.204124, 0.144338);
  \coordinate (v4_q) at (-0.353553, -0.204124, 0.144338);
  \coordinate (v5_q) at (0.0196419, 0.306186, 0.144338);
  \coordinate (v6_q) at (-0.107216, -0.0247398, 0.0366838);
  \coordinate (v7_q) at (-0.0749553, -0.0247398, 0.0366838);
  \coordinate (v8_q) at (0.020001, 0.182147, -0.00259064);
  \coordinate (v9_q) at (-0.0104344, 0.14289, 0.0366838);
  \coordinate (v10_q) at (0, 0, -0.433013);
  \coordinate (v11_q) at (0.118445, -0.019423, -0.0642586);

  \draw[color=white] (0,0,0) -- (.28,0,.14);

  \tikzstyle{linestyle_x} = [preaction={draw=white, line cap=round, line width=2.5 pt}, line cap=round, line join=round, color=myred, line width=1 pt]
  \tikzstyle{linestyle_p} = [preaction={draw=white, line cap=round, line width=2.5 pt}, line cap=round, line join=round, color=black, ultra thick]

  \draw[linestyle_p] (v0_q) -- (v10_q) -- (v4_q) -- cycle;
  \draw[linestyle_p] (v0_q) -- (v1_q) -- (v2_q) -- (v10_q);  
  \draw[linestyle_x] (v10_q) -- (v11_q) -- (v2_q) (v11_q) -- (v1_q); 
  \draw[linestyle_x] (v0_q) -- (v8_q) -- (v10_q) (v8_q) -- (v11_q);
  \draw[linestyle_x] (v0_q) -- (v9_q) -- (v8_q);
  \draw[linestyle_x] (v9_q) -- (v7_q) -- (v11_q);
  \draw[linestyle_x] (v10_q) -- (v6_q) -- (v9_q) (v7_q) -- (v6_q) -- (v4_q);
  \draw[linestyle_x] (v5_q) -- (v9_q) -- (v11_q);  
  \draw[linestyle_x] (v7_q) -- (v3_q);
  \draw[linestyle_p] (v1_q) -- (v5_q);
  \draw[linestyle_p] (v0_q) -- (v5_q) -- (v3_q) -- (v4_q);
  \draw[linestyle_p] (v2_q) -- (v3_q);
  
  \fill[pointcolor_p] (v0_q) circle (0.4 pt);
  \fill[pointcolor_p] (v1_q) circle (0.4 pt);
  \fill[pointcolor_p] (v2_q) circle (0.4 pt);
  \fill[pointcolor_p] (v3_q) circle (0.4 pt);
  \fill[pointcolor_p] (v4_q) circle (0.4 pt);
  \fill[pointcolor_p] (v5_q) circle (0.4 pt);
  \fill[pointcolor_x] (v6_q) circle (0.3 pt);
  \fill[pointcolor_x] (v7_q) circle (0.3 pt);
  \fill[pointcolor_x] (v8_q) circle (0.3 pt);
  \fill[pointcolor_x] (v9_q) circle (0.3 pt);
  \fill[pointcolor_p] (v10_q) circle (0.4 pt);
  \fill[pointcolor_x] (v11_q) circle (0.3 pt);
 
 \end{tikzpicture}
\end{subfigure}
\end{figure}
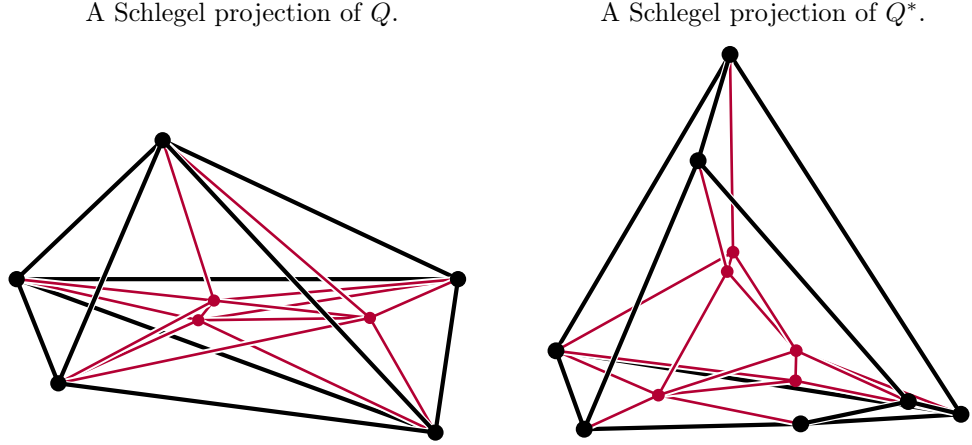
\setcounter{figure}{1}

An example of polytopes that are not determined by their graphs and their dimensions are \emph{cubical polytopes}.
These are polytopes in which every facet is a cube. Sanyal and Ziegler \cite[Corollary~ 3.8]{SanZie10} constructed combinatorially inequivalent, cubical $d$-polytopes on $n$ vertices with the same $\lfloor \tfrac{d}{2}\rfloor$-skeleton as the  $n$-cube.

Joswig \cite[Conjecture~3.1]{Joswig:2000} conjectured that cubical polytopes can be reconstructed from their dual graphs. Our first contribution is a counterexample (Example~\ref{ex:dualcubical}) to Joswig's conjecture in the strict sense of Problem~\ref{prob:rec}.
  
\begin{example}\label{ex:dualcubical} Consider the $4$-polytope $Q$ formed by the eight vertices:
\begin{align*}
	(-2, -2, -2, -2), && (-2, -2, -2,  4), && (-2, -2,  4, -2), && (-2,  4, -2, -2),\\
        ( 4,  4,  4, -2), && ( 4, -2, -2, -2), && ( 4, -5, -2,  4), && ( 1,-11, -1, 12)
\end{align*} 
The vertices are grouped such that nonadjacent vertices are on top of each other. The polytope $Q$ has the following twelve facets:
\begin{align*}
	      2x_2 +4x_3+ x_4&\geq -14&  x_1+2x_2+15x_3 &\geq -36\\
	 3x_1-3x_2 -4x_3-4x_4&\geq -8 & 14x_2+9x_4&\geq -46\\
	-2x_1+2x_2 -2x_3+ x_4&\geq -10& 3x_1+ x_2 &\geq -8\\
	  x_1- x_2  -x_3- x_4&\geq -2 & x_1 &\geq -2\\
	 -x_1-2x_2+ x_3-2x_4 &\geq -4 & x_3 &\geq -2\\
	-2x_1-2x_2+2x_3- x_4 &\geq -6 & x_4 &\geq -2\\
\end{align*}
Strictly speaking, each facet is obtained by intersecting the polytope $Q$ with the hyperplane derived by replacing the  inequality by the corresponding equality.   
It follows that the vertex-edge graph of $Q$ is isomorphic to the vertex-edge graph of a four dimensional cross-polytope, the polar of a $4$-cube. See additionally Figure~\ref{fig:p-cube-dual-graph}. As a consequence, the polar polytope $Q^*$ of $Q$ has a dual graph isomorphic to the dual graph of a $4$-cube. However, note that the polytope $Q^*$ has eight facets, each with six or seven vertices, and so $Q^*$ is noncubical.   
\end{example}

We have just shown the following.

\begin{proposition} Cubical polytopes are not reconstructible from their dual graph and dimension.
\end{proposition}

Example~\ref{ex:dualcubical} is our motivation for the following definition. We say that a polytope within a class of polytopes is \emph{class reconstructible} from its (dual) graph if first no other polytope in the class has the same (dual) graph, and second the face lattice of the  polytope can be reconstructed from its (dual) graph.  

\begin{problem}[Class reconstruction problem]\label{prob:class-rec} Given a class of polytopes, a parameter $d$, and the (dual) graph of a $d$-polytope in the class, determine the face lattice of the polytope.  
\end{problem}
 To the best of our knowledge, it is still open whether Joswig's conjecture is true in the sense of Problem~\ref{prob:class-rec}; that is, that cubical $d$-polytopes are class reconstructible from their dual graphs. There are partial results in that direction.
 Joswig proved that capped cubical polytopes, also known as stacked cubical polytopes, are reconstructible from their dual graphs and dimension within the class of cubical polytopes; see \cite[Theorem 3.7]{Joswig:2000}. Babson, Finschi and Fukuda showed in \cite{BabsonFinschiFukuda:2001} that cubical zonotopes are class reconstructible.
 
 We discuss Problem~\ref{prob:rec} and Problem~\ref{prob:class-rec} in detail for the class of matroid polytopes in Section~\ref{sec:RecMatroidPolytopes}. In the next section,  we introduce matroids and their polytopes,  and investigate properties of graphs of matroid polytopes. 

\section{Basis Exchange Graphs of Matroids}\label{sec:BasisExchangeGraphs}
\noindent
For the rest of this paper, we focus on matroids and their polytopes. A \emph{matroid}~$M$ is a nonempty collection $\cB$ of subsets of a finite set $E$, called  \emph{bases}, that satisfies the symmetric bases exchange property:
\begin{center}
	For every pair of bases $A,B\in\cB$ and each element $a\in A-B$ there exists an element $b\in B$ such that $A-a+b$ and $B+a-b$ are contained in $\cB$.
\end{center}
We use $\pm$ to denote inclusion or exclusion of elements from a set.
There are many ways to define matroids  \cite[Chapter~1]{Oxley:2011} and \cite[Appendix by Thomas Brylawski]{White:1986} for further cryptomorphisms. The finite set $E$ is the \emph{ground set} of the matroid, and two matroids are isomorphic if there is a bijection between their ground sets that maps bases to bases.

From now on, when the ground set $E$ of the matroid is not specified, we assume it is the set $[n]=\{1,2,\ldots,n\}$, where $n=\#E$.
For every basis $B\in \cB$, the \textit{indicator vector} $e_{B}\in \mathbb R^{n}$  is defined as $e_{B}=\sum_{i\in B}e_{i}$, where $e_{i}$ is the standard vector with a one in the $i$-coordinate. The {\it matroid (base) polytope} $P(M)$ associated with the matroid $M$ is the convex polytope
\begin{equation}
\label{eq:matroid-polytope}
    P(M)\coloneqq \conv\SetOf{e_B}{B\in\cB}.
\end{equation}

As mentioned above, one goal of this paper is to prove that matroid polytopes are class reconstructible from their graphs. 
In contrast to Problem~\ref{prob:rec} and Problem~\ref{prob:class-rec} where the input graph is understood to be the graph of a polytope, we do not make this requirement. Furthermore, we  do not even need the dimension $d$ as an input parameter. Given a graph, Algorithm~\ref{algo:label} in Section~\ref{sec:RecMatroidPolytopes} decides whether or not the graph is the graph of a matroid polytope. If it is, it returns a collection of bases of a matroid that correspond to the vertices of the unique matroid polytope with that graph; if is not, it returns \emph{false}.  

It is convenient to define the graph $G(M)$ of a matroid polytope $P(M)$ without resorting to the realisation of the polytope. The graph $G(M)$ can be defined on the set of bases $\cB$ of a matroid $M$. Two nodes $A,B\in\cB$ are adjacent if and only if $\size(A-B)=1$. We refer to $G(M)$ as the \emph{labelled basis exchange graph} whenever its nodes are labelled by the matroid bases and as the \emph{abstract basis exchange graph} if the labelling is unknown. However, we often identify nodes with their labels if they are given.

In this section, we determine under which conditions it is possible to read off the matroid~$M$ from the abstract basis exchange graph. This amounts to finding a node labelling by bases of $M$.
Section~\ref{sec:RecMatroidPolytopes} presents an algorithm that, given an abstract graph, computes these labels and therefore a matroid.

Let $M$ be a matroid on $E$ with bases $\cB$. The \textit{dual matroid} of $M$ is a matroid $M^{*}$ on the same ground set and with bases $\cB^{*} =\{E-B:B\in \cB\}$. A {\it loop} of a matroid is an element of the ground set that is contained in no basis and a {\it coloop} is an element of the ground set that is contained in every basis. 	Clearly, some matroids lead to isomorphic basis exchange graphs.

\newpage

\begin{lemma}\label{lem:samegraph} Two matroids $M$ and $M'$ have isomorphic basis exchange graphs if
   \begin{enumerate}
      \item \label{item:iso} $M'$ is isomorphic to $M$,
      \item \label{item:dual} $M$ is the dual of $M'$,
      \item \label{item:loops} $M'$ equals $M$ up to loops and coloops.
   \end{enumerate}
\end{lemma}
\begin{proof}
	\eqref{item:iso} If $M'$ and $M$ are isomorphic then there is a bijection $\varphi$ of the ground sets, which maps bases to bases. Obviously, this map induces a bijection between the nodes of $G(M')$ and $G(M)$.
   This is edge-preserving, as for any two bases $A$ and $B$ of $M'$ we have that $\varphi(A)-\varphi(B) =  \varphi(A-B)$ is of the same size as $A-B$.

	\eqref{item:dual} The map $B\mapsto E-B$ on the bases of $M$ is a bijection from the bases of $M$ to the bases of the dual matroid~$M^*$.
   Moreover, for each pair of bases $A$ and $B$ we have that $(E-A) - (E-B) = B - A$ and $\size(B-A) = \size(A-B)$.
   Thus, this map is a graph isomorphism.

	\eqref{item:loops} Adding a loop to a matroid does not change the set of bases,  and hence it preserves the basis exchange graph.
	The isomorphism of the basis exchange graphs for coloops follows  from \eqref{item:dual}.
\end{proof}

Next we describe an important operation which constructs a new matroid from two given matroids. For this purpose let $M_1$ be the matroid with bases $\cB_1$ on the ground set $E_1$ and and $M_2$ be the matroid with bases $\cB_2$ on $E_2$, where the ground sets $E_1$ and $E_2$ are disjoint. 
Let $E\coloneqq E_{1}\cup E_{2}$ and $\cB\coloneqq \SetOf{B_{1}\cup B_{2}}{B_{1}\in \cB_{1},B_{2}\in \cB_{2}}$.
The collection $\cB$ is the set of bases of a matroid $M$ on the ground set $E$ called the \emph{direct sum} $M_{1}\oplus M_{2}$ of $M_{1}$ and $M_{2}$. We say that a matroid is \emph{connected} if it is not the direct sum of two matroids. If a matroid $M=M_1\oplus\dots\oplus M_k$ is the direct sum of connected matroids, then  the ground sets $E_1,\ldots,E_k$ are the \emph{connected components of $M$}.

The matroid polytope $P(M_1\oplus M_2)$ is the \emph{product} $P(M_{1})\times P(M_{2})$ of the polytopes $P(M_{1})$ and $P(M_{2})$; that is for $M=M_1\oplus M_2$, 
\[
P(M)=P(M_{1})\times P(M_{2})\coloneqq\SetOf{(x_{1},x_{2})\in \mathbb R^{\size E_{1}+\size E_{2}}}{ x_{1}\in P(M_1) \text{ and } x_{2}\in P(M_2)}.
\] 
The basis exchange graph $G(M)$ is the \emph{Cartesian product} $G(M_{1})\times G(M_{2})$ of $G(M_{1})$ and $G(M_{2})$; that is, $G(M)$ is a graph on $V(G(M_{1}))\times V(G(M_{2}))$. Its edge set consists of all pairs $((v_1,v_2),(v_1',v_2'))$ of nodes with $v_1=v_1'$ and $(v_2,v_2')\in E(G(M_2))$, or $v_2=v_2'$ and $(v_1,v_1')\in E(G(M_1))$.  
Here $V(\cdot)$ and $E(\cdot)$ denote the node and edge set of a graph, respectively. The graphs $G(M_1)$ and $G(M_2)$ are \textit{factors} of the product $G(M)$.

It is convenient to define the  \emph{rank} of a matroid $M$, denoted $\rank M$, as the size of a basis in $\cB$ (which are all of same size) and its \emph{corank} as the rank of the dual matroid of $M$. We will later see in Example~\ref{ex:mat1} that we cannot read off the rank or the corank of $M$ from  the abstract exchange graph $G(M)$. 
\begin{example}\label{ex:directsum}
    Consider a rank-one matroid $M$ with bases $1$, $5$ and $6$ on the ground set $\{1,5,6\}$ and a rank-two matroid $M'$ with bases $23$ and $24$ and $34$ on $\{2,3,4\}$. Here and later, we omitted curly brackets and commas to denote bases to improve readability.
    The direct sum $M\oplus M'$ is a rank-three matroid on the ground set $\{1,2,3,4,5,6\}$ with the nine bases $123$, $124$, $134$, $236$, $235$, $246$, $245$, $346$ and $345$.
    The dual matroid $(M')^*$ of $M'$ is a rank-one matroid with bases $2$, $3$ and $4$. Another example of a  direct sum is the matroid $M\oplus (M')^*$, which is a rank-two matroid on the ground set $\{1,2,3,4,5,6\}$ with the nine bases $12$, $13$, $14$, $25$, $26$, $35$, $36$, $45$ and $46$.
    
    The two matroids $M\oplus M'$ and $M\oplus (M')^*$ have isomorphic basis exchange graphs; see Lemma~\ref{lem:disconnected} and Example~\ref{ex:mat1} below.
\end{example}

The following extends our list of matroids with isomorphic basis exchange graphs provided in Lemma~\ref{lem:samegraph}. 
\begin{lemma}\label{lem:disconnected}
   Let $M_1$, $M'_1$ and $M_2$, $M'_2$ be two pairs of matroids with isomorphic basis exchange graphs.   
   Then the basis exchange graph of $M_1\oplus M_2$ and $M'_1\oplus M'_2$ are isomorphic as well.
\end{lemma}
\begin{proof} The basis exchange graph $G(M_1\oplus M_2)$ is the Cartesian product of $G(M_1)$ and $G(M_2)$ and hence this graph is isomorphic to $G(M_1')\times G(M_2')=G(M'_1\oplus M'_2)$.
\end{proof}

Now we are prepared to exploit the structure of a basis exchange graph.
We benefit from results of Maurer, whose main theorem in \cite{Maurer:1973} is a characterisation of basis exchange graphs. A building block in his theorem is the concept of neighbourhood subgraphs.
Given a node $B$ in a graph $G$, its \emph{neighbourhood subgraph} $N(B)$ consists of all nodes that are adjacent to $B$ and all edges of $G$ that connect two of these nodes.  

Neighbourhood subgraphs in a basis exchange graph $G(M)$ are line graphs.
The \emph{line graph} $L(H)$ of a graph $H$ is the graph whose nodes are the edges of $H$ and two nodes share an edge whenever the corresponding edges in $H$ are adjacent. A graph is {\it bipartite} if its vertex set can be partitioned into two sets $X_{1}$ and $X_{2}$ with every edge connecting a node in $X_{1}$ to a node in $X_{2}$; we call the two sets, that label the nodes, $X_{1}$ and $X_{2}$ {\it partite sets}.

\begin{lemma}[{Link condition of basis exchange graphs. \cite[Lemma 1.8]{Maurer:1973}, \cite{ChalopinEtAl:2015}}] \label{lem:linegraph}
	For any node $B$ in a (labelled) basis exchange graph, the neighbourhood subgraph  is the line graph of a (labelled) bipartite graph with partite sets $B$ and $E-B$. 
\end{lemma}

Clearly, two isomorphic graphs have isomorphic line graphs. Also, the complete graph $K_3$ on three nodes and the connected bipartite graph $K_{1,3}$ with three edges pairwise adjacent have isomorphic line graphs. However, Whitney \cite{Whitney:1932} showed that these are the only nonisomorphic graphs with isomorphic line graphs. 
By \cref{lem:linegraph}, for a node $B$ of a basis exchange graph,  if we disregard the labels of $N(B)$, then Whitney's result yields that there is a unique abstract bipartite graph $G$ with line graph $N(B)$. We may assume that the nodes that form an edge of $G$ are ordered with respect to the partition sets. Then there is correspondence between the nodes of $N(B)$ and edges of $G$ via
\[
   u \enspace \mapsto  (x,y).
\]
We now assign node labels to $G$. For the node $u$ with label $C = B-b+e$, we label the node $x$ by the unique element $B-C = \{b\}$ and the node $y$ by the unique element in $C-B=\{e\}$. This labelling is well-defined as $N(B)$ is the line graph of $G$, and the inverse of the labelling map is given by $(b,e)\mapsto  B-b+e$. We denote the graph $G$ with or without its labelling by $L^{-1}(B)$.

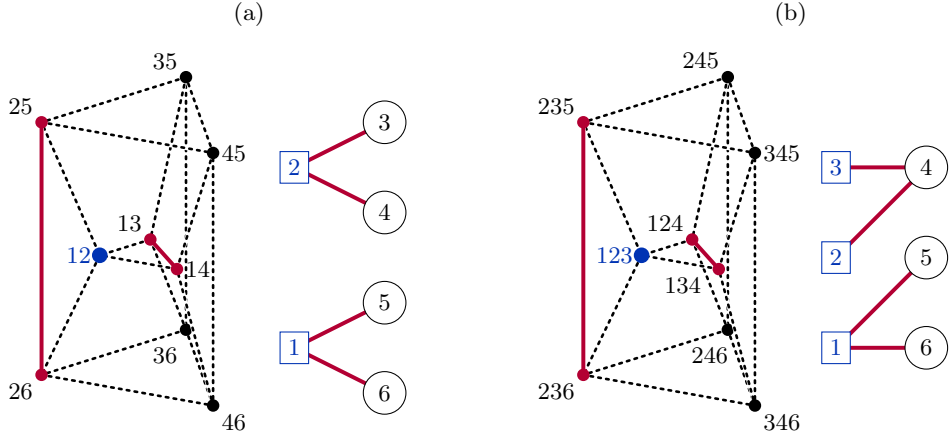
\begin{figure}
	\caption{Isomorphic basis exchange graphs of matroids of different ranks. 
	(a) The basis exchange graph of the matroid $M_1$ of Example~\ref{ex:mat1}, with the bipartite graph that corresponds to the neighbourhood subgraph $N(12)$ which is highlighted in red.
	(b) The basis exchange graph of the matroid $M_2$, with the bipartite graph of the subgraph $N(123)$ which is also highlighted in red.
	\label{fig:example}}
	\footnotesize
\begin{tikzpicture}[scale = 1.2,
                    color = {black}]

  \tikzstyle{linestyle} = [preaction={draw=white, line cap=round, line width=1.5 pt}, dotted, draw=myblack, line cap=round, line join=round,line width=1.0 pt];
  \tikzstyle{linestyle2} = [preaction={draw=white, line cap=round, line width=1.8 pt}, myred2, line cap=round, line join=round,line width=1.5 pt];
  \tikzstyle{inner} = [myblack, line width=1.1 pt,fill=myblack];
  \tikzstyle{node1} = [fill= white, draw=myblue, rectangle, text=myblue];
  \tikzstyle{node2} = [fill= white, draw, circle, minimum size=3mm];

  \coordinate (x0) at (0,0);
  \coordinate (x1) at (1,0.5);
  \coordinate (x2) at (1.3,-0.33);

  \coordinate (v) at ($0.15*(x1)+0.15*(x2)$);  
  \coordinate (n2) at ($(v)+0.45*(1.9,-0.34)$); 
  \coordinate (n1) at ($(v)+0.35*(1.6,0.5)$);   

  \coordinate (n3) at (-0.3,1.5);              
  \coordinate (o1) at (1.3,2);                 
  \coordinate (o3) at (1.6,1.16);              

  \coordinate (n4) at ($(n3)-(0,2.8)$);          
  \coordinate (o2) at ($(o1)-(0,2.8)$);          
  \coordinate (o4) at ($(o3)-(0,2.8)$);          

   \draw[linestyle] (o2) -- (n4) -- (o4) -- cycle;
   \draw[linestyle] (o1) -- (o2);
   \draw[linestyle] (o1) -- (n1) -- (o2);
   \draw[linestyle] (n1) -- (v) -- (n3);
   \draw[linestyle] (n2) -- (v) -- (n4);
   \draw[linestyle2] (n1) -- (n2);
   \draw[linestyle2] (n3) -- (n4);
     \draw[inner] (o2) circle (1.5 pt);
   \draw[linestyle] (o3) -- (n2) -- (o4) -- cycle;
   \draw[linestyle] (o1) -- (n3) -- (o3) -- cycle;

  \draw[inner,myblue] (v) circle (2 pt);
  \draw[inner, myred2] (n1) circle (1.5 pt);
  \draw[inner, myred2] (n2) circle (1.5 pt);
  \draw[inner, myred2] (n3) circle (1.5 pt);
  \draw[inner, myred2] (n4) circle (1.5 pt);
  \draw[inner] (o1) circle (1.5 pt);
  \draw[inner] (o3) circle (1.5 pt);
  \draw[inner] (o4) circle (1.5 pt);

   \node[left,myblue] at (v) {$12$};
   \node[above left] at (n1) {$13$};
   \node[right] at (n2) {$14$};
   \node[above left] at (n3) {$25$};
   \node[below left] at (n4) {$26$};
   \node[above left] at (o1) {$35$};
   \node[below left] at ($(o2)-(0,.1)$) {$36$};
   \node[right] at (o3) {$45$};
   \node[below right] at (o4) {$46$};

   \coordinate (a1) at (2.5, 1);
   \coordinate (a2) at (2.5,-1);
   \coordinate (b1) at (3.5, 1.5);
   \coordinate (b2) at (3.5, .5);
   \coordinate (b3) at (3.5,-.5);
   \coordinate (b4) at (3.5,-1.5);

   \draw[linestyle2] (b1) -- (a1) -- (b2);
   \draw[linestyle2] (b3) -- (a2) -- (b4);

  \node[node1] at (a1) {$2$};
  \node[node1] at (a2) {$1$};
  \node[node2] at (b1) {$3$};
  \node[node2] at (b2) {$4$};
  \node[node2] at (b3) {$5$};
  \node[node2] at (b4) {$6$};

  \coordinate (trans) at (6,0); 
  \coordinate (xx0) at ($(x0)+(trans)$);
  \coordinate (xx1) at ($(x1)+(trans)$);
  \coordinate (xx2) at ($(x2)+(trans)$);
  
    \coordinate (xv) at ($(v)+(trans)$);
    \coordinate (xn1) at ($(n1)+(trans)$);
    \coordinate (xn2) at ($(n2)+(trans)$);

    \coordinate (xn3) at ($(n3)+(trans)$);
    \coordinate (xo1) at ($(o1)+(trans)$);
    \coordinate (xo3) at ($(o3)+(trans)$);
            
    \coordinate (xn4) at ($(n4)+(trans)$);
    \coordinate (xo2) at ($(o2)+(trans)$);    
    \coordinate (xo4) at ($(o4)+(trans)$);

   \draw[linestyle] (xo2) -- (xn4) -- (xo4) -- cycle;
   \draw[linestyle] (xo1) -- (xo2);
   \draw[linestyle] (xo1) -- (xn1) -- (xo2);
   \draw[linestyle] (xn1) -- (xv) -- (xn3);
   \draw[linestyle] (xn2) -- (xv) -- (xn4);
   \draw[linestyle2] (xn1) -- (xn2);
   \draw[linestyle2] (xn3) -- (xn4);
     \draw[inner] (xo2) circle (1.5 pt);
   \draw[linestyle] (xo3) -- (xn2) -- (xo4) -- cycle;
   \draw[linestyle] (xo1) -- (xn3) -- (xo3) -- cycle;

  \draw[inner,myblue] (xv) circle (2 pt);
  \draw[inner, myred2] (xn1) circle (1.5 pt);
  \draw[inner, myred2] (xn2) circle (1.5 pt);
  \draw[inner, myred2] (xn3) circle (1.5 pt);
  \draw[inner, myred2] (xn4) circle (1.5 pt);
  \draw[inner] (xo1) circle (1.5 pt);
  \draw[inner] (xo3) circle (1.5 pt);
  \draw[inner] (xo4) circle (1.5 pt);

   \node[left,myblue] at (xv) {$123$};
   \node[above left] at (xn1) {$124$};
   \node[below left] at ($(xn2)-(.1,0)$) {$134$};
   \node[above left] at (xn3) {$235$};
   \node[below left] at (xn4) {$236$};
   \node[above left] at (xo1) {$245$};
   \node[below left] at ($(xo2)-(-.1,0.1)$) {$246$};
   \node[right] at (xo3) {$345$};
   \node[below right] at (xo4) {$346$};

   \coordinate (xa1) at (8.5, 1);
   \coordinate (xa2) at (8.5, 0);
   \coordinate (xa3) at (8.5, -1);
   \coordinate (xb1) at (9.5, 1);
   \coordinate (xb2) at (9.5, 0);
   \coordinate (xb3) at (9.5,-1);

   \draw[linestyle2] (xb2) -- (xa3) -- (xb3);
   \draw[linestyle2] (xa2) -- (xb1) -- (xa1);

  \node[node1] at (xa1) {$3$};
  \node[node1]  at (xa2) {$2$};
  \node[node1]  at (xa3) {$1$};
  \node[node2]  at (xb1) {$4$};
  \node[node2]  at (xb2) {$5$};
  \node[node2]  at (xb3) {$6$};


  \node at (2,2.7) {(a)};
  \node at (8,2.7) {(b)};

\end{tikzpicture}
\end{figure}

\begin{example} \label{ex:mat1}
	Consider the matroid $M_1=M\oplus (M')^*$ of Example~\ref{ex:directsum}. The neighbourhood subgraph of the basis $12$ in the basis exchange graph $G(M_1)$ contains the four nodes labelled by $13$, $14$, $25$, and $26$, and two nonadjacent edges. See Figure~\ref{fig:example}(a).
	Also consider the matroid $M_2=M\oplus M'$ of Example~\ref{ex:directsum}. Here, the neighbourhood subgraph of $123$ in the basis exchange graph $G(M_2)$ contains the four nodes $1 24$, $134$, $236$, and $235$, and two nonadjacent edges. See Figure~\ref{fig:example}(b).
   Figure~\ref{fig:example} illustrates that $G(M_1)$ is isomorphic to $G(M_2)$.
	Moreover, Figure~\ref{fig:example} shows that the two bipartite graphs $L^{-1}(12)$ and $L^{-1}(123)$, which  correspond to the isomorphic neighbourhood subgraphs $N(12)$ and $N(123)$ in Figure~\ref{fig:example}(a) and Figure~\ref{fig:example}(b), are also isomorphic. Additionally, this example shows that matroids of different ranks might have isomorphic basis exchange graphs; even if the matroids are loop and coloop-free, and not duals of one another.
\end{example}

In the rest of this section, we will prove that Lemma~\ref{lem:samegraph} and Lemma~\ref{lem:disconnected} cover all cases where matroids have isomorphic basis exchange graphs.

The next statement summarises the uniqueness of the rank and corank for connected matroids. In the proof, we use the following standard definition. For a matroid $M$ with  ground set $E$, the \emph{rank} $\rank(S)$ of a set $S\subseteq E$ is given by the maximum $\size (B\cap S)$, where $B$ ranges over all bases of $M$. 
\begin{lemma}\label{lem:connected}
    Let $M$ be a matroid and let $B$ be a basis of $M$.
   The following statements are equivalent:
   \begin{enumerate}
      \item \label{it:Mcon} $M$ is connected, 
      \item \label{it:Ncon} the neighbourhood subgraph $N(B)$ is connected,
      \item \label{it:Bcon} the bipartite graph $L^{-1}(B)$ is connected.
   \end{enumerate}
   Moreover, the connected (abstract) graph $L^{-1}(B)$ determines the rank and corank of $M$ as an unordered pair uniquely.
\end{lemma}
\begin{proof}
   Clearly, an edge path in the bipartite graph of $N(B)$ is mapped to a node path of $N(B)$.
   This already shows the equivalence of \eqref{it:Ncon} and \eqref{it:Bcon}.
   
   The matroid $M$ is disconnected if and only if $M = M_1\oplus M_2$, for some matroids $M_1$ and $M_2$.
   Let $E_1$ and $E_2$ be the ground sets of these matroids.
   For any $e\in E_i\cap B$, we have that $B-e+f$ is a basis only if $f\in E_i$ where $i\in\{1,2\}$. Hence there is no edge from a node in $E_1$ to one in $E_2$ in the bipartite graph. This proves that \eqref{it:Bcon} implies \eqref{it:Mcon}.
   
   On the other hand, suppose the bipartite graph $L^{-1}(B)$ is disconnected, and $S\subset E$ is the set of nodes of a connected component. Then, $S$ is a proper subset of $E$ and $\rank(S) = \rank(S\cap B)$; otherwise there would be an edge between a node in $S-B$ and a node in $B-S$. The same argument holds for $E-S$.
   Thus, the set $S$ satisfies $\rank(S) + \rank(E-S) = \size B = \rank(E)$. This implies that every basis of $M$ contains $\rank(S)$ elements of $S$, and thus $S$ is a connected component of $M$. Consequently, the matroid $M$ is disconnected.
   These arguments show that  \eqref{it:Bcon} follows from \eqref{it:Mcon}.

  Any choice of a labelling of a connected abstract bipartite graph $L^{-1}(B)$ determines the number of nodes in the two partite sets up to their interchange. These sizes are exactly the rank and corank of the matroid $M$ in some order.
\end{proof}
 The proof of \cref{lem:connected} gives the following corollary. 
\begin{corollary}\label{cor:connected} Let $M$ be a matroid and $B$ be a basis of $M$. A connected component of $M$ is the set of labels of all nodes in a connected component of the bipartite graph $L^{-1}(B)$. 
\end{corollary}

Suppose we have an abstract basis exchange graph with a marked node and the corresponding neighbourhood subgraph of the node. As we discussed above, any node labelling and partition of the nodes of the corresponding abstract bipartite graph into two partite sets lead to a labelling of the neighbourhood subgraph and the marked node. This labelling is unique whenever we pick an order of the two partite sets.
Next we describe how we can extend this labelling to a proper labelling of all nodes in the basis exchange graph.
The key ingredient is again the characterisation of Maurer.

A \emph{common neighbour subgraph} of a connected graph $G$ is a subgraph of $G$ that contains two nodes at distance two, all their common neighbours as additional nodes, and all the edges between the aforementioned nodes. The following is Maurer's characterisation.

\begin{theorem}[{\cite[Theorem 2.1]{Maurer:1973}}]\label{thm:Maurer}
  A connected graph $G$ is a basis exchange graph of a matroid if and only if
	\begin{enumerate}
		\item\label{it:Maurer1} it meets the interval condition, i.e., each common neighbour subgraph is a square, a square-based pyramid, or an octahedron;
		\item\label{it:Maurer2} it meets the positioning condition, i.e., for every node $B$ and every induced subgraph which is a square with nodes $A$, $C$, $D$, $E$ in cyclic order,  it holds that 
			\[\dist(B,A)+\dist(B,D)=\dist(B,C)+\dist(B,E);\; \text{and}\]
		\item\label{it:Maurer3} the neighbourhood subgraph $N(B)$ of some node $B$ is the line graph of a  bipartite graph.
	\end{enumerate}
\end{theorem}

The following is a direct consequence.
\begin{corollary}\label{cor:extlabelling}
   Any valid labelling of a node and its neighbourhood subgraph in a basis exchange graph has a unique extension to the entire graph.
\end{corollary}
\begin{proof} Let $G$ be the basis exchange graph of a matroid $M$ and let $v$ be a node of $G$. Let us further assume that $G$ is partially labelled. More precisely, we assume that $v$ and all nodes of distance at most $k\geq 1$ from $v$ are labelled by the corresponding bases of $M$.
If the node $u$ is at distance $k+1$ from $v$ then there exists a node $\hat v$ at distance $k-1$ from $v$ and at distance two from $u$. The node $\hat v$ is labelled, say, by $\ell(\hat v)$.
The common neighbour subgraph of $u$ and $\hat v$ contains a square by Theorem~\ref{thm:Maurer}(\ref{it:Maurer1}).
Moreover, we may choose the square such that its nodes are $u$, $\hat v$, and two other nodes $w_1$ and $w_2$. By construction, the nodes $w_1$, $w_2$ are at distance $k$ from $v$, and hence are labelled, say by $\ell(w_1)$ and $\ell(w_2)$, respectively. The basis exchange property of $M$ implies that the unique basis that corresponds to $u$ is $(\ell(w_1)\cap\ell(w_2)) \cup (\ell(w_1)\cup\ell(w_2)-\ell(\hat v))$.
Repetition of this procedure allows us to label all nodes at distance $k+1$ from $v$. Induction on $k$ proves the claim.
\end{proof}

\begin{figure}
	\caption{A partial labelling which extents uniquely to the entire graph
	\label{fig:exampleLabels}}
	\footnotesize
\begin{tikzpicture}[scale = 1.2,
                    color = {black}]

  \tikzstyle{linestyle} = [preaction={draw=white, line cap=round, line width=1.5 pt}, dotted, draw=black, line cap=round, line join=round,line width=1.0 pt];
  \tikzstyle{linestyle2} = [preaction={draw=white, line cap=round, line width=1.8 pt}, myred2, line cap=round, line join=round,line width=1.5 pt];
  \tikzstyle{inner} = [black, line width=1.1 pt,fill=black];
  \tikzstyle{inner2} = [black, line width=0.2 pt,fill=white];
  \tikzstyle{node1} = [fill= white, draw, rectangle];
  \tikzstyle{node2} = [fill= white, draw=none, inner sep = 1pt];

  \coordinate (x0) at (0,0);
  \coordinate (x1) at (1,0.5);
  \coordinate (x2) at (1.3,-0.33);

  \coordinate (v) at ($0.15*(x1)+0.15*(x2)$);  
  \coordinate (n2) at ($(v)+0.45*(1.9,-0.34)$); 
  \coordinate (n1) at ($(v)+0.35*(1.6,0.5)$);   
  
  \coordinate (n3) at (-0.3,1.5);              
  \coordinate (o1) at (1.3,2);                 
  \coordinate (o3) at (1.6,1.16);              

  \coordinate (n4) at ($(n3)-(0,2.8)$);          
  \coordinate (o2) at ($(o1)-(0,2.8)$);          
  \coordinate (o4) at ($(o3)-(0,2.8)$);          


   \draw[linestyle] (o2) -- (n4) -- (o4) -- cycle;
   \draw[linestyle] (o1) -- (o2);
    \draw[linestyle2] (o1) -- (n1);
   \draw[linestyle] (n1) -- (o2);
   \draw[linestyle2] (n1) -- (v) -- (n3);
   \draw[linestyle] (n2) -- (v) -- (n4);
   \draw[linestyle] (n1) -- (n2);
   \draw[linestyle] (n3) -- (n4);
     \draw[inner] (o2) circle (1.5 pt);
   \draw[linestyle] (o3) -- (n2) -- (o4) -- cycle;
   \draw[linestyle] (n3) -- (o3) -- (o1);
   \draw[linestyle2] (o1) -- (n3);

  \draw[inner, myred2] (v) circle (2 pt);
  \draw[inner, myred2] (n1) circle (2 pt);
  \draw[inner] (n2) circle (1.5 pt);
  \draw[inner, myred2] (n3) circle (2 pt);
  \draw[inner] (n4) circle (1.5 pt);
  \draw[inner, myred2] (o1) circle (2 pt);
  \draw[inner] (o3) circle (1.5 pt);
  \draw[inner] (o4) circle (1.5 pt);

   \node[left] at (v) {$123$};
   \node[above left] at (n1) {$124$};
   \node[below left] at ($(n2)-(0.1,0)$) {$134$};
   \node[above left] at (n3) {$235$};
   \node[below left] at (n4) {$236$};
   
   \node[above left] at (o1) {$\ell(u)$};

   \coordinate (l0) at (5.5, -1);
   \coordinate (l1a) at (4,0);
   \coordinate (l1b) at (7, 0);
   \coordinate (l2) at (5.5, 1);

   \draw[linestyle2] (l0) -- (l1a) -- (l2) -- (l1b) -- cycle;

  \draw[inner, myred2] (l0) circle (2 pt);
  \draw[inner, myred2] (l1a) circle (2 pt);
  \draw[inner, myred2] (l1b) circle (2 pt);
  \draw[inner, myred2] (l2) circle (2 pt);

  \node[below] at ($(l0)-(0,0.1)$) {$\ell(\hat v) = 123$};
  \node[below left] at ($(l1a)-(0.1,0)$) {$\ell(w_1) = 124$};
  \node[below right] at ($(l1b)+(0.1,0)$) {$\ell(w_2) = 235$};
  \node[above] at (l2) { $\ell(u)  = (124 \cap 235) \cup ((124\cap235)-123) = 245$};

\end{tikzpicture}
\end{figure}
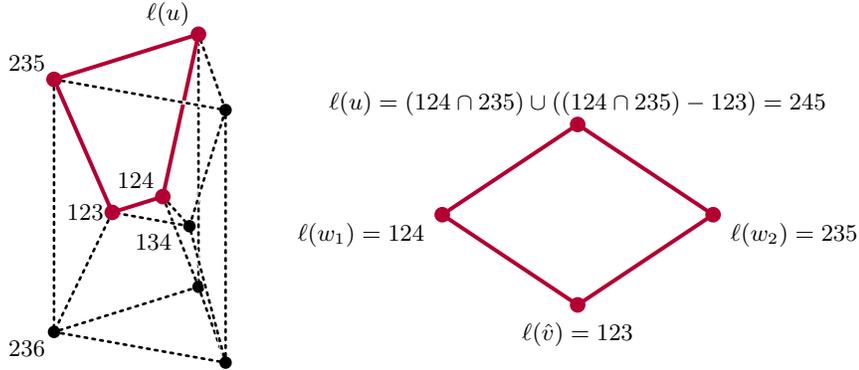

\begin{example} Consider the graph from Example~\ref{ex:mat1} with the partial labelling of a node and its neighbourhood subgraph as depicted on the left of Figure~\ref{fig:exampleLabels}. The highlighted square induces the unique label $\ell(u)$ of the node $u$, which is at distance two from the node labelled $123$. Here the node $123$ plays the role of both nodes $v$ and $\hat v$ in the proof of Corollary~\ref{cor:extlabelling}. How one obtains the label of $u$ is illustrated on the right hand side of Figure~\ref{fig:exampleLabels}. 
\end{example}

We sum up our results in the following theorem.
\begin{theorem}\label{thm:recmatroids}
   An abstract basis exchange graph determines its matroid $M$ up to
   \begin{enumerate}
      \item the isomorphic type of $M$,
      \item duality of the connected components of $M$, and
      \item possible loops and coloops.
   \end{enumerate}
\end{theorem}

\section{Reconstructing matroid polytopes}\label{sec:RecMatroidPolytopes}
\noindent In this section, we prove that matroid polytopes are class reconstructible from their graphs. Furthermore, we show that they are neither class reconstructible from their dual graphs nor reconstructible from their graphs. We begin by presenting Algorithm~\ref{algo:label}, which computes a (valid) labelling from a basis exchange graph if it exists.

\begin{algorithm}[htb]
  \dontprintsemicolon
	\Input{A connected graph $G=(V,E)$ with at least two nodes}
	\Output{A valid labelling $\ell$ or \texttt{False}}
  \smallskip

  Pick $v\in V$\;
  $H$ $\leftarrow \SetOf{w\in V}{(v,w)\in E}$\;
	\If{ $(H,E\cap (H\times H))$ is not a line graph of a bipartite graph}{
		\Return \texttt{False}\;
		}
  
  $(B_1 \cup B_2, E_B)$ $\leftarrow$ The ``original'' graph of the line graph $(H,E\cap (H\times H))$\;
  $\varphi:H\to E_B$ $\leftarrow$  The bijection of the above\;
  Set $\ell(v) = B_1$\;
  \For{$w\in H$}{
	  $(e,f)$ $\leftarrow$ $\varphi(w)$\;
	  Set $\ell(w) = B_1-e+f$ \;
  }
	\For{$j=2,\ldots,\size{B_1}$}{
	  $L$ $\leftarrow$ $\SetOf{u\in V}{\dist(u,v)=j}$\;
	  \While{ $L$ is not empty}{
		  Pick $u\in L$ and remove $u$ from $L$\;
		  Pick $\hat v\in\SetOf{x\in V}{\dist(x,v)=j-2 \text{ and } \dist(x,u)=2}$\;
		  $T$ $\leftarrow$ $\SetOf{w\in V}{\dist(w,u)=\dist(w,\hat v)=1}$\;
		  $W$ $\leftarrow$ $\SetOf{(x_1,x_2)\in T\times T}{x_1\neq x_2 \text{ and } (x_1,x_2)\not\in E}$\;
		  \If{$W$ is empty}{\Return \texttt{False}\;}
		  Pick $(w_1,w_2)\in W$\;
	  	  Set $\ell(u)=(\ell(w_1)\cap\ell(w_2)) \cup (\ell(w_1)\cup\ell(w_2)-\ell(\hat v))$\;
	  }
  }
	\For{$u,w\in V$}{
		\If{$\size (\ell(u)-\ell(v)) \neq \dist(u,v)$ \Or  $\dist(u,w)=2$ and the common neighbour subgraph is neither a square, square-based pyramid or octahedron}{\Return \texttt{False}\;}
	}
	\Return $\ell$\;

	\caption{Labelling a basis exchange graph.}
	\label{algo:label}

\end{algorithm}

\begin{proposition}
	Algorithm~\ref{algo:label} decides whether a graph with $n$ nodes is a basis exchange graph of a matroid, and if it is then the algorithm returns a valid labelling of the graph in $O(n^3)$ steps.
\end{proposition}
\begin{proof} The graph with exactly one node and no edges is a basis exchange graph of a matroid. To avoid the situation that there are no edges we restrict ourselves to the situation where $G$ is a connected graph with $n>1$ nodes.

Let us first prove that Algorithm~\ref{algo:label} stops with a correct result. Our analysis relies once more on Maurer's characterisation; see Theorem~\ref{thm:Maurer}. The graph $G$ violates condition (\ref{it:Maurer3}) if the procedure terminates in line 4.
The interval condition (\ref{it:Maurer1}) guarantees, if $G$ is a basis exchange graph, that there are two nodes $x_1$, $x_2$ in the common neighbour subgraph of $u$ and $\hat v$, such that the four nodes form a square. Hence, $(x_1,x_2)$ is a nonedge and $W$ is nonempty in that case. We conclude that the interval condition is violated if the procedure terminates in line 19.
Clearly, the algorithm stops at line 24 when either a common neighbour subgraph violates (\ref{it:Maurer1}) or the labelling is incorrect.
If the graph $G$ is a basis exchange graph, then the node $v$ and its neighbourhood have a valid labelling, and by Corollary~\ref{cor:extlabelling} the labelling has a unique extension that is given by the exchange property, which is reflected in line 21. Thus, if the node labelling of $G$ is incorrect then $G$ is not a basis exchange graph of a matroid.

Our next goal is to show that $G$ is the basis exchange graph of a matroid with bases $\smallSetOf{\ell(v)}{v\in V}$ whenever the algorithm terminates in line 25. The algorithm stops in line 24 or earlier whenever $G$ violates (\ref{it:Maurer1}), and in line 4 if (\ref{it:Maurer3}) is violated.
Notice that a labelling is valid if the size of the set $\ell(v)-\ell(w)$ is one if and only if $(v,w)$ is an edge of $G$. In that case the labelling reflects distances via
\[
    \dist(u,v) = \size(\ell(v)-\ell(u)) \enspace.
\]
It is easy to see that valid labels of nodes in a square are of the form
\[
    I+a+b,\, I+a+c,\, I+b+d,\, I+c+d
\]
for some set $I$ and elements $a,b,c,d$. With the above formula for the distances we get that the algorithm stops at line 24 if (\ref{it:Maurer2}) is violated. Consequently, we conclude that the output is correct.

We complete the proof of correctness by showing that the algorithm terminates. Moreover, we show that the algorithm takes at most  $O(n^3)$ steps in a worst case scenario. For this analysis we may assume that $G$ is given as an adjacency matrix. This matrix is of size $n^2$ whenever $G$ has $n$ nodes, which is also an upper bound for the number of edges. Thus the second line takes less than $n^2$ steps.
In the  third line, the algorithm verifies whether the graph $(H,E\cap(H\times H))$  is a line graph of a graph $(B,E_B)$. 
This graph has $\size H < n$ nodes and $m < \size E < n^2$ edges. Already in 1973, Roussopoulos  \cite{Roussopoulos:1973} presented an $O(\max\{\size H, m\})$ algorithm that recognises a given line graph and reconstructs the original graph $(B,E_B)$.
This original graph $(B,E_B)$ is unique unless the line graph is a triangle; see the discussion in Section~\ref{sec:BasisExchangeGraphs}.
In this case, we choose the original graph to be the bipartite graph  $K_{1,3}$ and not the complete graph $K_3$. 
Hence lines 3--5 take $O(n^2)$ time.
Here we also need to verify whether $(B,E_B)$ is a bipartite graph; this can be done by traversing the graph, e.g., via depth-first search.
One either finds a two-colouring of its nodes, in the case that the graph is bipartite, or a cycle of odd length, if the graph is not bipartite. 
Traversing the graph also takes at most $O(n^2)$ steps. Moreover, the two colour classes $B_1$, $B_2$ form a partition of the nodes $B$ of the original graph.
We may assume that no node of $(B,E_B)$ is isolated and thus $\size B_1 \leq \size E_B = \size H < n$.
Assigning labels to $v$ and all its neighbours can be done in $O(n^2)$ time.

It is efficient to compute all pairs of nodes at distance two in advance, before using these pairs in lines 15 and 23.
There are many ways to do so, e.g., by applying Floyd and Warshall's algorithm, which runs in $O(n^3)$.
Moreover, we store additionally the distances from $v$ to all other nodes.

The two nested loops in lines 11 and 13 go around $\size B_1-1<n $ and $\size L \leq n$ times, respectively. That is, the inner loop, lines 14--21, 
 iterates at most $n^2$ times. Initialising the list $L$ in line 12 takes linear time, while taking $u$ and removing it is constant.
Finding the node $\hat v$ in line 15 can be done in $O(n)$ steps by using the distance information we computed in advance.
The set $T$, which is the intersection of two sets, can be computed in $O(n)$ steps, e.g., by iterating once through the nodes and picking those which are adjacent to $u$ and $\hat v$. The graph $G$ is not a graph of a matroid if $T$ has more than four elements. Thus we may assume that the instructions in lines 17 to 20 take constant time. Finally, computing the label $\ell(u)$ takes a linear amount of time. Thus, the inner loop is linear in $n$, and so lines 11--21 run in $O(n^3)$ time.

There are fewer than $n^2$ pairs of nodes of distance two. Hence this is an upper bound for the number of iterations  of the loop in line 22. Computing the common neighbour subgraph is essentially the same as computing $T$ in line 16. This graph has at most six nodes, thus checking if the graph is a square, a square-based pyramid, or an octahedron is constant in time. All this takes $O(n^3)$ time.
Independent of that loop, one might iterate through all pairs of nodes to verify that the labelling is correct, which can also  be done in $O(n^3)$ time. We conclude that the algorithm stops after at most $O(n^3)$ steps, with the correct result.
\end{proof}

\begin{remark}
	Note that a (valid) labelling of a basis exchange graph gives a matroid. More precisely, the vertex labels of the graph are the collection of bases of a matroid with this basis exchange graph.
	In Section~\ref{sec:BasisExchangeGraphs}, we saw that this matroid is not unique. The matroids with the same graph are characterised by Theorem~\ref{thm:recmatroids}. 
	Moreover, the indicator vectors of the labels or bases form the vertices of the matroid polytope. Thus, we might be a bit imprecise  by saying that Algorithm~\ref{algo:label} computes a polytope. It is remarkable that the dimension of the polytope is not an input parameter of the algorithm, as it is customary in the reconstruction problems (see Problems  \ref{prob:rec} and \ref{prob:class-rec}). 
\end{remark}

Strictly speaking, Algorithm~\ref{algo:label} computes the vertices of a matroid polytope from its graph. This representation of the polytope is sufficient for the analysis below.
For the sake of completeness, we wish to discuss how one could compute the facets of the polytope. This analysis requires some additional definitions. A \emph{circuit} of a matroid $M$ is a subset $C$ of the ground set such that $C$ is not a subset of a basis, but every proper subset of $C$ is contained in some basis. A \emph{cocircuit} of $M$ is a circuit of the dual matroid $M^*$. The \emph{closure} of a set $S$ in $M$ is the largest set $F\supseteq S$ with $\rank(S)=\rank(F)$. A set $S$ is \emph{closed} if it equals its closure.

Uniform matroids and hypersimplices are essential in our description of facets of matroid polytopes. A \emph{uniform matroid} $U_{r,n}$ of rank $r$ on $n$ elements has $\tbinom{n}{r}$ bases. That is, every $r$-set forms a basis.
Its matroid polytope is the \emph{hypersimplex}
\[
    \Delta(r,n) \coloneqq P(U_{r,n}) = \SetOf{x\in[0,1]^n}{\sum_{i=1}^n x_i = r}.
\]
If $M$ is another matroid of rank $r$ on $n$ elements, then its polytope $P(M)$ is a subpolytope of $\Delta(r,n)$ and
the graph $G(M)$ is a subgraph of the graph $G(U_{r,n})$  \cite{GelfandEtAl:1987}.

It follows from \cite[Proposition 7 and 13]{JoswigSchroeter:2017} that every facet of a matroid polytope of a connected matroid $M$ is induced by one of the following three inequalities:
\begin{equation}
\label{eq:facet-matroid}
x_i\geq 0,\;   x_i\leq 1,\; \text{or}    \sum_{i\in F} x_i \leq \rank(F)\;   \text{for some closed set $F$ which is a union of circuits.}
\end{equation}
Again, this means that each facet is obtained by intersecting the polytope $P(M)$ with the hyperplane derived by replacing the  inequality by the corresponding equality.  
A subpolytope of the hypersimplex is a matroid polytope if and only if its graph consists solely of edges of the hypersimplex; see \cite[Theorem 4.1]{GelfandEtAl:1987}. This implies that faces and facets of a matroid polytope are matroid polytopes. 
Moreover, the dimension of the matroid polytope $P(M)$ is the difference between the size of the ground set $E$ and the number of connected components of $M$.
It follows that the facet $\smallSetOf{x\in P(M)}{ \sum_{i\in F} x_i = \rank(F)}$ is a matroid polytope of a matroid with connected components $F$ and $E-F$; see \cite[Theorem 3.2 and 3.4]{Fujishige:1984} or \cite[Proposition 2.4 and 2.6]{FeichtnerSturmfels:2005}.

\begin{proposition}
Given a labelled exchange graph of a loop-free matroid $M$, there is an algorithm that computes all circuits of $M$ and the facets of the matroid polytope $P(M)$.
\end{proposition}

\begin{proof}
    The bases of $M$ are given as a labelling of the basis exchange graph $G$. And the ground set $E$ of the matroid $M$ is the union of all its bases as $M$ is loop-free. 
    Next we explain how to compute the circuits and closures of sets. 
    
    Let $B$ be a basis of $M$ and $e\notin B$. Then the neighbourhood $N$ of $e$ in the bipartite graph  $L_M^{-1}(B)$ is a subset of $B$. Moreover, the set $N+e$ is a circuit, because no basis of $M$ contains $N+e$ but every proper subset of it is contained in a basis. The latter follows directly from the definition of the neighbourhood subgraph of $B$ and the fact that $N$ is the neighbourhood of $e$ in the bipartite graph $L_M^{-1}(B)$. For the former, let us assume that $B'$ is a basis that contains $N+e$. We may also assume that $\size (B\cap B') = \size B -1$; otherwise exchange repeatedly elements in $B'-(N+e)$ with elements of $B$. However, in this situation a basis exchange between $B$ and $B'$ would imply that there is an edge between $e$ and some node in $E-N$, contradicting that $N$ is the neighbourhood of $e$ in $L_M^{-1}(B)$. Every circuit $C$ is of the described form for some basis $B$ with $\size (B\cap C) = \rank(C)=  \size C-1$. Thus we may enumerate all circuits of $M$ by computing the neighbourhoods of nodes in $L_M^{-1}(B)$, where $B$ varies over all bases of $M$.
    
   To obtain the closure of a set $S\subset E$, we take a basis $B$ with $\rank(S) = \size(B\cap S)$; here the rank of $S$ can be determined as the maximum size of all intersections $B' \cap S$ where $B'$ varies over all bases of $M$. Next, we compute $L^{-1}(B)$,  and then consider all elements of $E-B$ that do not share an edge with an element in $B-S$ in the bipartite graph $L_M^{-1}(B)$. These elements together with $S$ form the closure $F$ of $S$, because $\rank(F)=\size(B\cap F) = \size(B\cap S) = \rank(S)$, and $\rank(F+e)>\rank(F)$ for all $e\in E-F$. The latter follows from the fact that, for each $e\in E-F$, either $e\in B-F$, or $e\in E-B-F$ and $e$ is adjacent to at least one element in $B-F$.

    Now we verify which inequalities in \eqref{eq:facet-matroid} give rise to facets.
    Let us assume that $M$ is connected; otherwise we will apply the following to every connected component of the matroid $M$, which can be obtained from $L_M^{-1}(B)$; see
    Corollary~\ref{cor:connected}. 
    
    For each nonempty closed set $F\subsetneq E$ that is a union of circuits, we compute its rank,  and then check if $Q=\smallSetOf{x\in P(M)}{ \sum_{i\in F} x_i = \rank(F)}$ is a facet as follows. In any case, the set $Q$ is a face of the polytope $P(M)$ as the inequality $\sum_{i\in F} x_i \leq \rank(F)$ holds for all $x\in P(M)$. Now, $Q$ is the matroid polytope of a matroid $M'$ whose bases are all the bases of $M$ satisfying the additional equation. 
    For a basis $B$ of $M'$, we consider the labelled bipartite graphs $L_M^{-1}(B)$ and $L_{M'}^{-1}(B)$ corresponding to the node labelled $B$ in $M$ and $M'$, respectively. Removing the edges between the nodes in $F$ and $E-F$ in the bipartite graph $L_M^{-1}(B)$ results in the bipartite graph $L_{M'}^{-1}(B)$.
    This is because these edges are precisely those that correspond to bases $B'$ in the neighbourhood subgraph $N(B)$ that violate the condition $\size(B\cap B') = \rank(F)$. Now $Q$ is a facet if and only if $L_{M'}^{-1}(B)$ has precisely one connected component more than $L_M^{-1}(B)$.
    This is so because the number of connected components of $L_{M'}^{-1}(B)$ is the number of connected components of $M'$, which follows from \cref{cor:connected}, and thus corresponds to the codimension of $P(M')$.
    
    We continue with the inequalities $x_i\geq 0$ for $i\in E$. In this case, let us choose a node $v$ labelled by a basis $B$ which does not contain $i$. This node exists because otherwise $i$ would be a loop, but $M$ is loop-free. Similar to the previous case, the set $Q=\smallSetOf{x\in P(M)}{ x_i = 0}$ is a face of $P(M)$, as every point in $P(M)$ satisfies $x_i\ge 0$. This face is a matroid polytope of a matroid $M'$ whose bases are all bases of $M$ that do not contain $i$; in particular, this includes the basis $B$.
    As before, the bipartite graph $L_{M'}^{-1}(B)$ is obtained from the connected graph $L_M^{-1}(B)$ by removing the edges incident to $i$. The face $Q$ is a facet if the number of connected components of $L_{M'}^{-1}(B)$ is two. 
    
    Similarly, we test inequalities $x_i\leq 1$ for each $i\in E$. In this case, we choose instead a basis $B$ that contains $i$. Then we count the number of connected components of $L_{M}^{-1}(B)$ and the subgraph obtained by removing the edges incident to $i$, which corresponds to the matroid of the face $P(M)\cap\{x_i=1\}$. In this way, we are able to decide which inequalities  in \eqref{eq:facet-matroid} give rise to facets of $P(M)$.
\end{proof}

\begin{figure}
	\caption{The graph that leads to the matroid $M$ of Example~\ref{ex:Facet} and the bipartite graph $L^{-1}(135)$. From this graph one can read off the closure of the circuit $123$, which is $1234$. \label{fig:exampleFacets}}
	\footnotesize
\begin{tikzpicture}[scale = 1.2,
                    color = {black}]

  \tikzstyle{linestyle} = [preaction={draw=white, line cap=round, line width=1.5 pt}, draw=black, line cap=round, line join=round,line width=1.0 pt];
  \tikzstyle{linestyle2} = [preaction={draw=white, line cap=round, line width=1.8 pt}, myred2, line cap=round, line join=round,line width=1.5 pt];
  \tikzstyle{inner} = [black, line width=1.1 pt,fill=black];
  \tikzstyle{inner2} = [black, line width=0.2 pt,fill=white];
  \tikzstyle{node1} = [fill= white, text=myblue, draw, rectangle];
  \tikzstyle{node2} = [fill= white, draw, circle, minimum size=3mm];
  \tikzstyle{node3} = [fill= white, draw=none, circle, inner sep = 0pt];

   \coordinate (a) at (5.5, -1);
   \coordinate (d) at (4,0);
   \coordinate (b) at (7, 0);
   \coordinate (c) at (5.5, 1);

   \draw[linestyle, myblue]  (d) -- (a);
   \draw[linestyle]  (a) -- (b);
   \draw[linestyle]  (a) to[bend right=30] (c);
   \draw[linestyle, myblue]  (a) to[bend left=30] (c);
   \draw[linestyle, myblue]  (c) to[bend left=30] (b);
   \draw[linestyle]  (c) to[bend right=30] (d);
      
  \draw[inner] (a) circle (2 pt);
  \draw[inner] (b) circle (2 pt);
  \draw[inner] (c) circle (2 pt);
  \draw[inner] (d) circle (2 pt);
  
  \node[below left, node3, text=myblue] at ($0.5*(a)+0.5*(d)$) {$1$};
  \node[below right, node3] at ($0.5*(a)+0.5*(b)$) {$6$};
  \node[right, node3] at ($0.5*(b)+0.5*(d)$) {$4$};
  \node[left, node3, text=myblue] at ($0.5*(b)+0.5*(d)$) {$3$};
  \node[node3, text=myblue] at ($0.5*(b)+0.5*(c)$) {$5$};
  \node[node3] at ($0.5*(c)+0.5*(d)$) {$2$};

   \coordinate (a3) at (9.5, 1);
   \coordinate (a1) at (9.5, 0);
   \coordinate (a5) at (9.5, -1);
   \coordinate (b4) at (10.5, 1);
   \coordinate (b2) at (10.5, 0);
   \coordinate (b6) at (10.5,-1);
  
   \draw [rounded corners,fill=gray!20, draw=none] (9.2,1.3)--(10.8,1.3)--(10.8,-0.3) -- (9.2,-0.3) --cycle;
 
   \draw[linestyle2] (a1) -- (b4) -- (a3) (a1);
   \draw[linestyle, line width=1.5 pt] (a1) -- (b2);
   \draw[linestyle] (a5) -- (b6) -- (a1);

  \node[node1, draw=myred] at (a1) {$3$};
  \node[node2]  at (b2) {$4$};
  \node[node1,draw=myred]  at (a3) {$1$};
  \node[node2, draw=myred]  at (b4) {$2$};
  \node[node1]  at (a5) {$5$};
  \node[node2]  at (b6) {$6$};

\end{tikzpicture}
\end{figure}
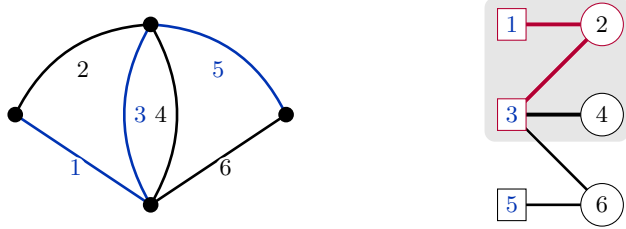

\begin{example}\label{ex:Facet} Consider the graphical matroid $M$ of the edge labelled graph that can be found on the left of Figure~\ref{fig:exampleFacets}. A basis of this matroid is the set of edges in a spanning tree.
This matroid contains the circuit $C=123$ and basis $B=135$.
The closure $F$ of $C$ is the set $1234$ and the inequality $x_1+x_2+x_3+x_4 \leq 2$ is facet-defining for the matroid polytope $P(M)$ of $M$. This can be read off the bipartite graph $L^{-1}(B)$, which is illustrated on the right of Figure~\ref{fig:exampleFacets}. You might also read off that $x_3+x_4+x_5+x_6 \leq 2$ is a facet-defining inequality by considering the circuit $356$. Besides, it can be verified that the remaining facets of $P(M)$ are given by $x_i\geq 0$ where $i\in\{3,4\}$ and  $x_i\leq 1$ where $i\in\{1,2,5,6\}$. This example also demonstrates that the circuit closure $34$ does not induce a facet, even if $x_3+x_4\leq 1$ is a valid inequality, because the induced subgraph of $L^{-1}(135)$ with nodes $1,2,5$ and $6$ is disconnected.
\end{example}

Before we return to the reconstruction problem, we finish the discussion about the facets of a matroid polytope with the following computational question.
\begin{question}
Is it possible to enumerate all facets of a matroid polytope $P(M)$ from its graph in a polynomial number of steps, measured in the number of bases of $M$?
\end{question}

\begin{theorem}\label{thm:rec-matroid-polys}
   The (abstract) vertex-edge graph of a matroid polytope determines the dimension and combinatorial type of the polytope 
   uniquely in the class of matroid polytopes.
\end{theorem}
\begin{proof}
	Given the abstract vertex-edge graph of a matroid polytope or the basis exchange graph of a matroid,  we are able to reconstruct the matroid up to the three points given in Theorem~\ref{thm:recmatroids}.
	Now we will show that isomorphic matroids, matroids with dual connected components, and matroids with loops and coloops lead to matroid polytopes with the same combinatorial type.

	Clearly, an isomorphism of matroids translates into a coordinate permutation of the space where the matroid polytope lies.

	The matroid polytope of a connected matroid is mapped to the matroid polytope of the dual matroid via the coordinate transformation $x_i\mapsto 1-x_i$, and hence they have the same combinatorial type.
	The matroid polytope of a direct sum is the product of the matroid polytopes of the summands.
	This implies that matroids with dual connected components have combinatorial isomorphic matroid polytopes. 

	A loop in a matroid fixes a coordinate to be $0$, and a coloop to be $1$. In other words, in both cases the matroid polytope is embedded in a hyperplane.
	
	We conclude that the vertex-edge graph  uniquely determines the combinatorial type of its matroid-polytope.  
\end{proof}

Theorem~\ref{thm:rec-matroid-polys} proves that matroid polytopes are class reconstructible from their graphs, in the sense of Problem~\ref{prob:class-rec}, via Algorithm~\ref{algo:label}. However, matroid polytopes are neither reconstructible nor class reconstructible from their dual graphs. This is illustrated best by the matroid polytopes of uniform matroids.

Now we are ready to provide some counterexamples.
\begin{proposition}\label{prop:dualgraphs} Let $k,r,n$ be integers such that $1 \leq k < r < n$ and $k < n-r$.
	The $(k-1)$-skeleton of the polar polytope of the hypersimplex $\Delta(r,n)$ is isomorphic to the $(k-1)$-skeleton of the $n$-dimensional cross-polytope.
\end{proposition}

\begin{proof} The hypersimplex $\Delta(r,n)$ is the intersection of the hyperplane $\sum_{i=1}^n x_i = r$ with the $n$-dimensional $0/1$-cube, and all its facets are induced by those of the cube. Now, let $F$ be an $(n-m)$-dimensional face of that cube. Without loss of generality, we may assume that 
	\[F = F_1\cap \cdots \cap F_\ell \cap F_{\ell+1} \cap \cdots \cap F_m,\] where
	$F_i = \SetOf{x\in[0,1]^n}{x_i=1}$ for $i\leq\ell$ and
	$F_i = \SetOf{x\in[0,1]^n}{x_i=0}$ for $\ell < i\leq m$.
	
	The face $F$ is an $(n-m)$-dimensional $0/1$-cube, and $F$ intersected with $\sum x_i = r$ is isomorphic to the hypersimplex $\Delta(r-\ell,n-m)$. This is an $(n-m-1)$-dimensional polytope, unless $r=\ell$ or $m+r=n+\ell$. In these two cases, the hypersimplex is  a single point.
	It follows that the intersection of the facets $F_i\cap \smallSetOf{x\in\RR^n}{\sum x_i=r}$ of the hypersimplex forms the $(n-m-1)$-face $F\cap \smallSetOf{x\in\RR^n}{\sum x_i=r}$,  whenever $\ell<r$ and $r-\ell<n-m$. Clearly $\ell\leq m$ and $r-\ell\leq r$. Hence, for every $m<\min\{r,n-r\}$,  the intersection of $m$ of these facets yields a $(n-m-1)$-face of $\Delta(r,n)$.
	
From the observation that every face of the hypersimplex is of that form, follows that the codimension-$k$ faces of $\Delta(r,n)$  are in bijection with the codimension-$k$ faces of the $n$-dimensional cube, whenever $0\leq k\leq m$. Moreover, this induces an inclusion-preserving bijection between the faces of the hypersimplex and those of the cube whose codimension is at most $m$. As the $n$-dimensional cube is the polar of the $n$-dimensional cross-polytope;  our claim is now established.
\end{proof}

The next corollary makes explicit what we just showed in Proposition \ref{prop:dualgraphs}, namely that hypersimplicies, and therefore also matroid polytopes, are not class-reconstructible from their dual graphs.
\begin{corollary}\label{cor:rec-dual-matroid-polys}
The dual graphs of the $(n-1)$-dimensional polytopes $\Delta(r,n)$ and $\Delta(r+1,n)$ are isomorphic whenever $n-2 > r > 2$.
\end{corollary}

We proceed with a discussion on the reconstructibility of matroid polytopes as in Problem~\ref{prob:rec}. We begin this discussion with the following classical example that we mentioned in Section~\ref{sec:intro-rec}.
\begin{example}
	Consider the cyclic polytope of dimension $d\geq 4$ and $n>d$ vertices. Its vertex-edge graph is the complete graph with $n$ nodes \cite[Chapter 12]{Gruenbaum:2003}. This graph is also the vertex-edge graph of an $(n-1)$-dimensional simplex, which is the hypersimplex $\Delta(1,n)$, and thus the matroid polytope of an uniform matroid on $n$ elements of rank $1$. 
\end{example}

\begin{figure}[t!]
	\caption{Schlegel projections of the two $4$-polytopes $P(M)$ and $Q$ of Example~\ref{ex:counterexample}. 
	On the right, the nonmatroidal facet of $Q$ is exposed. } 
\label{fig:counterexample}
\begin{subfigure}[t]{0.45\textwidth}
    \centering
    \subcaption*{A Schlegel projection of $P(M)$.}

\begin{tikzpicture}[x  = {(-0.7cm,-0.6cm)},
                    y  = {(-0.0cm,0.5cm)},
                    z  = {(0.7cm,-0.6cm)},
                    scale = 3.5,
                    color = {lightgray}]


  \tikzstyle{pointcolor_p} = [fill=black]
  \tikzstyle{pointcolor_x} = [fill=myred]
  
  \coordinate (v0_q) at (0.5, 0.7, 0.5);
  \coordinate (v1_q) at (1, 1, 0);
  \coordinate (v2_q) at (0.5, 0.7, 0.1);
  \coordinate (v3_q) at (0.5, 0.3, 0.5);
  \coordinate (v4_q) at (1, 0, 0);
  \coordinate (v5_q) at (0, 1, 1);
  \coordinate (v6_q) at (0.1, 0.7, 0.5);
  \coordinate (v7_q) at (0, 1.1, 0);
  \coordinate (v8_q) at (0, 0, 1);

  \tikzstyle{linestyle_x} = [preaction={draw=white, line cap=round, line width=2.5 pt}, line cap=round, line join=round, color=myred,  line width=1]
  \tikzstyle{linestyle_p} = [preaction={draw=white, line cap=round, line width=2.5 pt}, line cap=round, line join=round, color=black, ultra thick]

  \draw[linestyle_p] (v5_q) -- (v1_q);
  \draw[linestyle_x] (v1_q) -- (v0_q);
  \draw[linestyle_x] (v2_q) -- (v0_q);
  \draw[linestyle_x] (v2_q) -- (v1_q);
  \draw[linestyle_x] (v3_q) -- (v0_q);
  \draw[linestyle_x] (v3_q) -- (v2_q);
  \draw[linestyle_p] (v4_q) -- (v1_q);
  \draw[linestyle_x] (v4_q) -- (v2_q);
  \draw[linestyle_x] (v4_q) -- (v3_q);
  \draw[linestyle_x] (v5_q) -- (v0_q);
  \draw[linestyle_x] (v6_q) -- (v0_q);

  \fill[pointcolor_x] (v0_q) circle (0.8 pt);
  \node at (v0_q) [text=black, inner sep=6pt, above, draw=none, align=left] {\footnotesize $34$};

  \draw[linestyle_x] (v6_q) -- (v2_q);
  \draw[linestyle_x] (v6_q) -- (v3_q);
  \draw[linestyle_x] (v6_q) -- (v5_q);
  \draw[linestyle_p] (v7_q) -- (v1_q);

  \fill[pointcolor_p] (v1_q) circle (1 pt);
  \node at (v1_q) [text=black, inner sep=4pt, above left, draw=none, align=left] {\footnotesize $24$};

  \draw[linestyle_x] (v7_q) -- (v2_q);

  \fill[pointcolor_x] (v2_q) circle (0.8 pt);
  \node at (v2_q) [text=black, inner sep=4pt, above left, draw=none, align=left] {\footnotesize $23$};

  \draw[linestyle_p] (v7_q) -- (v4_q);
  \draw[linestyle_p] (v7_q) -- (v5_q);
  \draw[linestyle_x] (v7_q) -- (v6_q);
  \draw[linestyle_x] (v8_q) -- (v3_q);

  \fill[pointcolor_x] (v3_q) circle (0.8 pt);
  \node at (v3_q) [text=black, inner sep=6pt, below, draw=none, align=left] {\footnotesize $35$};

  \draw[linestyle_p] (v8_q) -- (v4_q);

  \fill[pointcolor_p] (v4_q) circle (1 pt);
  \node at (v4_q) [text=black, inner sep=4pt, below left, draw=none, align=left] {\footnotesize $25$};

  \draw[linestyle_p] (v8_q) -- (v5_q);

  \fill[pointcolor_p] (v5_q) circle (1 pt);
  \node at (v5_q) [text=black, inner sep=4pt, above right, draw=none, align=left] {\footnotesize $14$};

  \draw[linestyle_x] (v8_q) -- (v6_q);

  \fill[pointcolor_x] (v6_q) circle (0.8 pt);
  \node at (v6_q) [text=black, inner sep=4pt, above right, draw=none, align=left] {\footnotesize $13$};

  \draw[linestyle_p] (v8_q) -- (v7_q);

  \fill[pointcolor_p] (v8_q) circle (1 pt);
  \node at (v8_q) [text=black, inner sep=4pt, below right, draw=none, align=left] {\footnotesize $15$};
  \fill[pointcolor_p] (v7_q) circle (1 pt);
  \node at (v7_q) [text=black, inner sep=5pt, above, draw=none, align=left] {\footnotesize $12$};

\end{tikzpicture}
\end{subfigure}
\begin{subfigure}[t]{0.45\textwidth}
    \centering
    \subcaption*{A Schlegel projection of $Q$.}

\begin{tikzpicture}[x  = {(-0.7cm,-0.6cm)},
                    y  = {(-0.0cm,0.5cm)},
                    z  = {(0.7cm,-0.6cm)},
                    scale = 3.5,
                    color = {lightgray}]


  \tikzstyle{pointcolor_p} = [fill=black]
  \tikzstyle{pointcolor_x} = [fill=myred]
  
  \coordinate (v0_q) at (0.5, 0.7, 0.5);
  \coordinate (v1_q) at (1, 1, 0);
  \coordinate (v2_q) at (0.75, 0.6, 0.25);
  \coordinate (v3_q) at (0.5, 0.3, 0.5);
  \coordinate (v4_q) at (1, 0, 0);
  \coordinate (v5_q) at (0, 1, 1);
  \coordinate (v6_q) at (0.1, 0.7, 0.5);
  \coordinate (v7_q) at (0, 1.1, 0);
  \coordinate (v8_q) at (0, 0, 1);

  \tikzstyle{linestyle_x} = [preaction={draw=white, line cap=round, line width=2.5 pt}, line cap=round, line join=round, color=myred3, line width=1, dotted]
  \tikzstyle{linestyle_p} = [preaction={draw=white, line cap=round, line width=2.5 pt}, line cap=round, line join=round, color=myred4, ultra thick, dotted]

  \tikzstyle{linestyle_xf} = [preaction={draw=white, line cap=round, line width=2.5 pt}, line cap=round, line join=round, color=myred, line width=1]
  \tikzstyle{linestyle_pf} = [preaction={draw=white, line cap=round, line width=2.5 pt}, line cap=round, line join=round, color=black, ultra thick]

  \draw[linestyle_pf] (v5_q) -- (v1_q);
  \draw[linestyle_xf] (v1_q) -- (v0_q);
  \draw[linestyle_xf] (v2_q) -- (v0_q);
  \draw[linestyle_xf] (v2_q) -- (v1_q);
  \draw[linestyle_xf] (v3_q) -- (v0_q);
  \draw[linestyle_xf] (v3_q) -- (v2_q);
  \draw[linestyle_pf] (v4_q) -- (v1_q);
  \draw[linestyle_xf] (v4_q) -- (v2_q);
  \draw[linestyle_xf] (v4_q) -- (v3_q);
  \draw[linestyle_xf] (v5_q) -- (v0_q);
  \draw[linestyle_x] (v6_q) -- (v0_q);

  \fill[pointcolor_x] (v0_q) circle (0.8 pt);
  \node at (v0_q) [text=black, inner sep=4pt, below left, draw=none, align=left] {\footnotesize $34$};

  \draw[linestyle_x] (v6_q) -- (v2_q);
  \draw[linestyle_x] (v6_q) -- (v3_q);
  \draw[linestyle_x] (v6_q) -- (v5_q);
  \draw[linestyle_p] (v7_q) -- (v1_q);

  \fill[pointcolor_p] (v1_q) circle (1 pt);
  \node at (v1_q) [text=black, inner sep=4pt, above left, draw=none, align=left] {\footnotesize $24$};

  \draw[linestyle_x] (v7_q) -- (v2_q);

  \fill[pointcolor_x] (v2_q) circle (0.8 pt);
  \node at (v2_q) [text=black, inner sep=6pt, below, draw=none, align=left] {\footnotesize $w$};

  \draw[linestyle_p] (v7_q) -- (v4_q);
  \draw[linestyle_p] (v7_q) -- (v5_q);
  \draw[linestyle_x] (v7_q) -- (v6_q);
  \draw[linestyle_xf] (v8_q) -- (v3_q);

  \fill[pointcolor_x] (v3_q) circle (0.8 pt);
  \node at (v3_q) [text=black, inner sep=6pt, below, draw=none, align=left] {\footnotesize $35$};

  \draw[linestyle_pf] (v8_q) -- (v4_q);

  \fill[pointcolor_p] (v4_q) circle (1 pt);
  \node at (v4_q) [text=black, inner sep=4pt, below left, draw=none, align=left] {\footnotesize $25$};

  \draw[linestyle_pf] (v8_q) -- (v5_q);

  \fill[pointcolor_p] (v5_q) circle (1 pt);
  \node at (v5_q) [text=black, inner sep=4pt, above right, draw=none, align=left] {\footnotesize $14$};

  \draw[linestyle_x] (v8_q) -- (v6_q);

  \fill[pointcolor_x, myred3] (v6_q) circle (0.8 pt);
  \node at (v6_q) [text=black, inner sep=4pt, above right, draw=none, align=left] {\footnotesize $13$};

  \draw[linestyle_p] (v8_q) -- (v7_q);

  \fill[pointcolor_p] (v8_q) circle (1 pt);
  \node at (v8_q) [text=black, inner sep=4pt, below right, draw=none, align=left] {\footnotesize $15$};
  \fill[pointcolor_p, myred4] (v7_q) circle (1 pt);
  \node at (v7_q) [text=black, inner sep=5pt, above, draw=none, align=left] {\footnotesize $12$};

\end{tikzpicture}
\end{subfigure}
\end{figure}
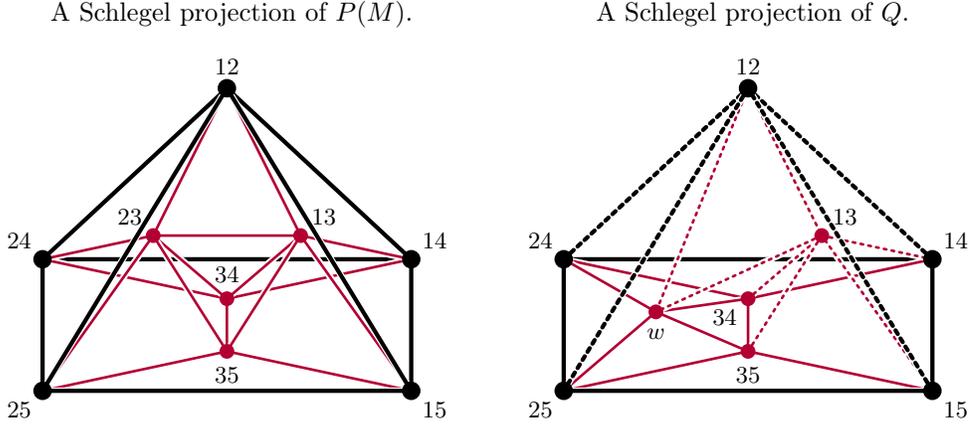

However, the next example illustrates that matroid polytopes are not reconstructible from their graphs even if if the dimension is given, i.e., they are not reconstructible in the sense of Problem~\ref{prob:rec}.

\goodbreak
\begin{example}\label{ex:counterexample} Consider the matroid polytope $P(M)$ of the rank-two matroid $M$ on five elements with the nine bases
\[
12,\,13,\,14,\,15,\,23,\,24,\,25,\,34,\,35.
\]
That is, the elements $4$ and $5$ are parallel. Let us denote by $v_{ij} = e_i+e_j$ the vertices of $P(M)$.
The polytope $P(M)$ has nine facets, three tetrahedra, three  square-based pyramids, one octahedra, and one prism $T$ over a triangle. This prism is the convex hull of the six points $v_{14}$, $v_{15}$, $v_{24}$, $v_{25}$, $v_{34}$ and $v_{35}$.
A Schlegel projection of the polytope is shown in Figure~\ref{fig:counterexample}.

Now consider the $4$-polytope $Q$ which is derived from $P(M)$ by replacing the vertex $v_{23}= (0,1,1,0,0)$ by the point $w = (-\tfrac{1}{2},\tfrac{1}{2},1,\tfrac{1}{2},\tfrac{1}{2})$, i.e., $Q$ is the convex hull of the points $v_{12}$, $v_{13}$, $v_{14}$, $v_{15}$, $v_{24}$, $v_{25}$, $v_{34}$, $v_{35}$ and $w$. The point $w$ is a vertex of $Q$ and lies in the affine hull of the prism $T$ of $P(M)$. These seven points form a nonmatroidal facet of $Q$.
All together, the polytope $Q$ has nine vertices, and ten facets: three tetrahedra, six square-based pyramids, and the aforementioned nonmatroidal facet. A Schlegel projection is shown in Figure~\ref{fig:counterexample}.

It is straight forward to check that the two polytopes $P(M)$ and $Q$ have isomorphic graphs.
\end{example}

The above example is minimal in the sense that matroid polytopes with eight or fewer vertices are reconstructible from their graphs and dimension. This can be  verified by checking all polytopes with eight or fewer vertices \cite{MiyataMoriyamaFukuda:2013}. 
We end this article with a last example that shows that hypersimplices are not reconstructible from their graphs, and with two related open questions.
\begin{example}\label{ex:counterexamplehypersimplex} Consider the hypersimplex $\Delta(2,5)$, which is a 4-dimensional polytope with ten vertices and ten facets. Five of these facets are tetrahedra and the other five are octahedra.

Let $Q$ be the polytope that is derived from $\Delta(2,5)$ by pushing the vertex $(1,1,0,0,0)$ in direction of the centre, say to the position  $w=(\tfrac{5}{8},\tfrac{5}{8},\tfrac{1}{4},\tfrac{1}{4},\tfrac{1}{4})$. Then, three of the octahedral facets of $\Delta(2,5)$ break into two square-based pyramids each. Thus, the polytope $Q$ is the convex hull of ten points. It has  $13$ facets: 
five tetrahedra, six square-based  pyramids, and two octahedra. It is not a matroid polytope.
Besides, the graphs of $\Delta(2,5)$ and $Q$ are isomorphic.
\end{example}

From Example~\ref{ex:counterexamplehypersimplex}, we derive the following statement.
\begin{proposition}\label{prop:counterexample} Hypersimplices, and hence matroid polytopes, are not reconstructible from their graphs and dimension.
\end{proposition}

We saw that matroid polytopes and hypersimplices are not reconstructible from their graphs, i.e., their $1$-skeleta. It would be interesting to find answers to the following questions.
\begin{question} What is the smallest $k$ depending on $r$ and $n$ such that the hypersimplex $\Delta(r,n)$ of dimension $n-1$ is reconstructible from its $k$-skeleton and dimension? 
\end{question}
 
\begin{question} What is the smallest $k$ such that every $n$-dimensional matroid polytope (of rank $r$) is reconstructible from its $k$-skeleton and dimension? 
\end{question}

\section*{Acknowledgements}
\noindent
The authors thank Michael Joswig for his comments about the reconstruction of cubical polytopes. Furthermore, they are grateful for the detailed feedback from two anonymous referees that helped to improve the manuscript. 
The second author is supported by the Knut and Alice Wallenberg Foundation, and also thanks Federation University Australia for the hospitality during a two-week stay.

\bibliographystyle{alpha}
\bibliography{References}

\end{document}